\documentclass[11pt,epsfig,amsfonts]{amsart}
\usepackage{mathrsfs} 

\usepackage{epsfig}
\usepackage{amsmath}
\usepackage{amssymb}
\usepackage{amscd}
 \usepackage{graphicx}
\usepackage{color}
\usepackage{changes}

\newtheorem{thm}{Theorem}[section]

\newtheorem{lem}[thm]{Lemma}
\newtheorem{prop}[thm]{Proposition}

\newtheorem{cor}[thm]{Corollary}
\newtheorem{de}[thm]{Definition} 

\newcommand{\mb}{\mathbb}

\def \t{\theta}

\def \p{\phi}

\def \o{\omega}

\numberwithin{equation}{section}

\begin{document}


\thanks{This work is partially supported by grants from NSFC (11225105, 11331007, 11431012) and NSF (1413603)}


\author{Wen Huang} \address[Wen Huang] {School of Mathematics, Sichuan University, Chengdu 610064, Sichuan, P. R. China\\ and
 School of Mathematical Sciences\\
    University of Science and Technology of China \\Hefei 230026,
Anhui, P. R. China}
\email[W.~Huang]{wenh@mail.ustc.edu.cn}

\author{Kening Lu} \address[Kening Lu] {Department of Mathematics\\
 Brigham Young University\\
 Provo, Utah 84602, USA}
\email[k.~Lu]{klu@math.byu.edu}

\title[Entropy, Chaos and  weak Horseshoe for Infinite Dimensional Random Dynamical Systems]{Entropy, Chaos and weak Horseshoe
for Infinite Dimensional Random Dynamical Systems}

\pagestyle{plain}

\begin{abstract} {In this paper, we study the complicated dynamics of infinite dimensional random dynamical systems which include deterministic dynamical systems as their special cases in a Polish space. Without assuming any hyperbolicity, we proved if a continuous random map has a positive topological entropy,  then it contains a  topological horseshoe. We also show that the positive topological entropy implies the chaos in the sense of Li-Yorke. The complicated behavior exhibiting here is induced by the positive entropy but not the randomness of the system.
} \end{abstract}

\maketitle

\section{Intoduction}

Entropy plays an important role in the study of behavior of dynamical systems. It measures the rate of increase in dynamical complexity as the system evolves with time. The measure-theoretic entropy was introduced in 1950's by Kolmogorov \cite{Kol} and Sinai \cite{Sinai2} for studying measurable dynamical systems. Sinai \cite{Sinai} studied an ergodic measure preserving automorphism  $f$ of a Lebesgue space $(X, \mu)$  and proved that if the measure-theoretic entropy of $f$ is positive, then $f$ contains factor automorphisms  which are isomorphic to Bernoulli shifts.

The topological entropy was first introduced in 1965 by Adler, Konheim and McAndrew for studying dynamical systems in topological spaces. In metric spaces a different definition of topological entropy was introduced by Bowen in 1971 and independently Dinaburg in 1970. A fundamental problem is to  characterize the chaotic behavior of orbits of a $C^k$ ($k\geq 0$) dynamical  system $f$ topologically or geometrically (in terms of horseshoe) in the presence of positive topological entropy.

In his remarkable paper \cite{K}, A. Katok proved that for a measure preserving hyperbolic $C^2$ diffeomorphism on a compact Riemannian manifold, the positive entropy implies the existence of a Smale horseshoe.

Without assuming any hyperbolicity, Blanchard, Glasner, Kolyada, and Maass \cite{BGKM}
showed that for a  homeomorphism on a compact metric space $X$, the positive topological entropy implies the chaos in the sense of Li-Yorke \cite{LY},
i.e., there is a subset $S$ of $X$, which is a union of countably many Cantor sets, and $\kappa>0$ such that
for every pair $x_1,x_2$ of distinct points in $S$, the following holds.
\begin{align*} &\liminf_{n\rightarrow +\infty} d(f^n(x_1),f^n(x_2))=0, \\
&\limsup_{n\rightarrow +\infty} d(f^n(x_1),f^n(x_2))\ge \kappa.
\end{align*}
However, this result does not yield the existence of a horseshoe.

Recently, Lian and Young obtained remarkable results on the implication of positive entropy for infinite dimensional deterministic dynamical systems. In \cite{LianY1}, they extended Katok's results to  $C^2$ differentiable maps with a nonuniformly hyperbolic compact invariant set supported by an invariant measure in a separable Hilbert space. In their second paper \cite{LianY2}, Lian and Young  went much further and studied a $C^2$ semiflow  in a Hilbert space and proved that if it has a nonuniformly hyperbolic compact invariant set supported by an invariant measure, then the positive entropy implies the existence of horseshoes. In this case, the semiflow may have one simple zero Lyapunov exponent, which implies that the associated time-one map restricted to this invariant set is partially hyperbolic with one-dimensional center direction. This result is new even for flow generated by ordinary differential equations.

The proofs of the results obtained by Katok and Lian-Young rely on not only the positive of entropy but also the hyperbolic geometric structures of systems. The horseshoes are constructed by using stable and unstable manifolds.

\smallskip

In present paper, we study $C^0$ infinite dimensional random dynamical systems which include deterministic dynamical systems as their special cases in a Polish space. Without assuming any hyperbolicity, we proved if a continuous random map has a positive topological entropy,  then it contains a  topological horseshoe. We also show that the positive topological entropy implies the chaos in the sense of Li-Yorke. The complicated behavior exhibiting here is induced by the positive entropy but not the randomness of the system.

Since there is no any hyperbolic geometric structure available, we take a different approach. We use Rohlin's theory of Lebesques systems as a basic tool and utilize the disintegration of measures, Pinsker algebra, entropy, and ergodic theory. To overcome the obstacle due to lack of hyperbolicity, we construct an ``independent'' partition with an equal conditional probability measure $\bar \mu_y$ for almost all $y$ via the disintegration of a measure relative to a factor. This partition is the key for constructing the horseshoe. Other challenges are: (i) the infinite dimensional dynamical systems generated, for example, by parabolic PDEs are not invertible and (ii) the phase space is not
locally compact.

\smallskip

Let $(\Omega, \mathcal{F}, P)$ be a probability space
and $(\theta^n)_{n\in \mathbb{Z}}$ be a measurable ${P}$-measure
preserving dynamical system on $\Omega$. A discrete time
random dynamical system (or a cocycle) on a metric space $X$ over
the dynamical system ${\theta}^n$ is a measurable map
$$\phi(n,\cdot,\cdot):\Omega \times X \to X, \quad (\omega, x) \mapsto\phi(n, \omega, x), \quad\text{for } n\in\mathbb{Z}^+$$
such that the map $\phi(n,\omega):=\phi(n,\omega,\cdot)$ forms a
cocycle over $\theta^n$:
\[
\phi(0, \omega)=Id, \quad \hbox{ for all }\; \omega \in \Omega,
\]
\[
\phi(n+m,\omega)=\phi(n,\theta^{m}\omega)\phi(m,\omega), \quad
\hbox{ for all }\; m, n \in \mathbb{Z}^+, \quad\omega \in \Omega.
\]
When $\phi(n, \omega, \cdot): X \to X$ is
continuous, $\phi(n, \omega, x)$ is called a continuous random
dynamical system. Replacing $\mathbb{Z}$ and $\mathbb{Z}^+$ by $\mathbb{R}$ and $\mathbb{R}^+$ repectively gives a continuous time random dynamical system, see Section 2 for details. \smallskip

\smallskip
A typical example of random dynamical systems is the solution operator for a stochastic
differential equation:
\[
d x_t= f_0(x_t) dt + \sum_{k=1}^d f_k(x_t) \circ dB^k_t
\]
where $x\in \mathbb{R}^d$, $f_k, 0\leq k\leq d$, are smooth vector
fields, and $B_t=(B_t^1, \cdots, B_t^d)$ is  the standard
d-dimensional Brownian motion defined on the probability space
$(\Omega, \mathcal{F}, {P})$ and $d B^k_t$ is the
Stratonovich differential. Here, $(\Omega, \mathcal{F},
\mathbb{P})$ is the classic Wiener space, i.e., $\Omega =\{
\omega\,:\, \omega(\cdot) \in C(\mathbb{R}, \mathbb{R}^d),
\omega(0)=0\}$ endowed with the open compact topology so that $\Omega$ is a Polish space and
$\mathbb{P}$ is the Wiener measure. Define a measurable dynamical
system $\theta^t$ on the probability space
$(\Omega, \mathcal{F}, {P})$ by the Wiener shift
$(\theta^t\omega)(\cdot)=\omega(t+\cdot)-\omega(t)$ for $t>0$.  It
is well-known that ${P}$ is invariant and ergodic under
$\theta^t$.
This
measurable dynamical system $\theta^t$ is also
called a metric dynamical system. It models
the noise of the system.

\smallskip

We consider a random set  ${K}\in {\mathcal F}\otimes\mathcal B(X)$ and  use  $K(\o)$ to denote its $\o$-section $\{x\in X\;|\;( \o,x)\in K\}$. $K$ is said to be invariant under $\phi(n,\omega)(x)$ if
for all $n\in \mb Z^+$
\[\p(n, \o)K(\o)\subset K( \t^n \o)\quad  P-a.s..\]
Examples of  such random invariant sets in applications are the global random attractors of dissipative stochastic partial differential equations.



\smallskip

We study the complicated dynamics of infinite dimensional random dynamical systems restricted to random invariant sets.  We assume that $(\Omega, \mathcal{F},P, \theta)$ is a Polish system (see Section 3) and the phase space $X$ is a Polish space with the distance function $d$. We consider a continuous random dynamical system $\phi(n, \omega, x)$ and write the time-one map $\phi(1, \omega, x)$ as  $\phi(\omega)(x):=\phi(1, \omega,
x)$. Then $\phi(\omega)$ is the so-called random map.  This random map generates the random dynamical system:
\[
\phi(n, \omega, x)=
\begin{cases} \phi(\theta^{n-1}\omega)\cdots
\phi(\omega)(x), & n >0, \\
I, & n=0. \\
\end{cases}
\]
Our main result can be stated as follows.

\medskip
\noindent {\bf Main Theorem.}  {\it  Let $\phi$ be an injective continuous random dynamical system on a Polish space $X$  over an ergodic Polish system $(\Omega, \mathcal{F},P,\theta)$. Let ${K}$ be a random $\phi$-invariant set with compact $\omega$-section ${K}(\omega)$ such that $\phi(\omega)(K(\omega))=K(\theta \omega)$.    If the topological  entropy is positive, i.e.,
\[h_{\text{top}}(\phi,\mathcal{K})>0,
\]
then
\begin{itemize}
\item[(i)] the dynamics of $\phi$ restricted to ${K}$ is chaotic;
\item[(ii)] $\phi$ restricted to ${K}$ has a weak horseshoe of two symbols;
\item[(iii)] in addition, if $(\Omega, \mathcal{F}, {P},\theta)$ is a
compact metric system and $K(\omega),\,\omega\in\Omega$ is  a strongly compact random set, then $\phi$ restricted to ${K}$ has a full horseshoe of two symbols.
\end{itemize}
}

\smallskip

 {$\phi$ has  a weak horseshoe of two symbols if there exist  subsets $U_1,U_2$ of $X$ such that the following
properties hold
\begin{enumerate}
\item $U_1$ and $U_2$ are non-empty, bounded,  and closed  and
$d(U_1,U_2)>0$.

\item there is a constant $b>0$ satisfying for $P$-a.e. $\omega\in \Omega$,
there exists $M_{b,\omega}\in \mathbb{N}$ such that for any natural
number $m\ge M_{b,\omega}$, there is a subset $J_m\subset
\{0,1,2,\cdots,m-1\}$ with $|J_m|\ge bm$ (positive density), and for any $s\in
\{1,2\}^{J_m}$, there exists $x_s\in K(\omega)$ with
$\phi(j,\omega,x_s)\in U_{s(j)}\cap K(\theta^j \omega)$ for any
$j\in J_m$.
\end{enumerate}}

\smallskip
By a {\it full horseshoe of two symbols} we mean that there exist  subsets $U_1,U_2$ of $X$  such that the following
properties hold
\begin{enumerate}
\item $U_1$ and $U_2$ are non-empty, bounded, and closed subsets of $X$  and
$d(U_1,U_2)>0$.

\item  there is  a constant $b>0$ satisfying for $P$-a.e. $\omega\in \Omega$,
there exists $J(\omega)\subset  \mathbb{N}_0$ such that the limit
$\lim_{m\rightarrow +\infty}\frac{1}{m}|J(\omega)\cap \{0,1,2,\cdots, m-1\}|$ exists and is larger than or equal to $b$
(positive density), and for any $s\in
\{1,2\}^{J(\omega)}$, there exists $x_s\in K(\omega)$ with
$\phi(j,\omega,x_s)\in U_{s(j)}\cap K(\theta^j \omega)$ for any
$j\in J(\omega)$.
\end{enumerate}
\smallskip
The horseshoe here is an extension of Smale's horseshoe. The main difference is that the time spent by  the orbit $\phi(j,\omega,x_s)$  bouncing between $U_1$ and $U_2$ is nonuniform.

\smallskip
We point out that the random dynamical systems $\phi$ generated by both random parabolic PDEs and random wave equations are continuous and injective.

\smallskip

For deterministic dynamical systems, we have the following result ( see \cite{HY} for topological dynamical systems in  compact metric spaces and \cite{KL-2006} for $C^*$-dynamics).
\smallskip

\noindent {\bf Corollary.}  {\it  Let $f$ be an injective continuous map on a Polish space $X$.  Let ${K}$ be a compact invariant set of $f$.    If the topological  entropy is positive, i.e.,
\[h_{\text{top}}(f,\mathcal{K})>0,
\]
then $f|_K$ has a full horseshoe of two symbols.}

\medskip

Random dynamical systems arise in the modelling of many phenomena in physics, biology, climatology, economics, etc. when uncertainties or random influences, called noises, are taken into account. These random effects are not only introduced to compensate for the defects in some deterministic models, but also are often rather intrinsic phenomena. The need for studying random dynamical systems was pointed out by Ulam and von Neumann \cite{UvN} in 1945. It has flourished since the 1980's due to the discovery that stochastic ordinary differential equations generate finite dimensional random dynamical systems through the efforts of Harris, Elworthy, Baxendale, Bismut, Ikeda, Kunita, Watanabe, and others. Random dynamical systems are nonuniform in nature in terms of hyperbolicity. There is an extensive literature on the nonuniformly hyperbolic theory and the ergodic theory for both independent and identically distributed random transformations and stationary random dynamical systems, which we refer to Arnold \cite{A}, Kifer \cite{K86, Kifer88, K01},  Ledrappier and Young \cite{LeY1, LeY2},  Liu and Qian \cite{LQ}, Liu \cite{L-p1}, Kifer and Liu \cite{KL},  and the references therein.


\smallskip

 The study of infinite dimensional random dynamical systems was initiated by Ruelle in \cite{Ruelle1, Ruelle2}  where the Oseledets' multiplicative ergodic theorem and Pesin's stable manifold theorem were established in a Hilbert space, and the notion of random attractor was introduced. Infinite dimensional random dynamical systems are usually  generated by stochastic partial differential equations (SPDEs) and contain randomness in many ways, such as stochastic forcing, uncertain parameters, random sources or inputs, and random initial and boundary conditions.  There is a vast amount of works on the ergodic theory, the existence of random attractors  and the theory of invariant manifold. We do not attempt to give an exhaustive list of references. Results on ergodic theory can be found in Mane \cite{Mane}, Thieullen \cite{Thieu}, Schauml\"offel and Flandoli \cite{SchFla}, Da Prato and Zabczyk \cite{DaPZab96}, E, Khanin, Mazel, and Sinai \cite{EKMS}, Hairer and Mattingly \cite{HM}, Lian and Lu \cite{LLu}. For the existence of random attractors we refer to    Crauel, Debussche, and Flandoli \cite{CDF},
 Crauel and Flandoli \cite{cr-fl}, Schmalfuss \cite{Schm97c}, and Bates, Lu, and Wang \cite{BLW}. Theory of random invariant manifolds can be found in Da Prato and Debussche \cite{DaPDeb96}, Mohammed and Scheutzow \cite{MoS},
 Duan, Lu and Schmalfuss \cite{DLS}, and Mohammed, Zhang, and Zhao \cite{MZZ}. The problem we study here is about the complicated dynamical behavior on random invariant sets such random attractors.


\smallskip

We organize the paper as follows. In Section 2, we introduce basic concepts concerning random
dynamical systems and random invariant sets and state  our main results. In Section 3, we review some of basic concepts and results from measurable dynamical systems and introduce some basic lemmas. The proof of the main result is given in Section 4 and Section 5.

\section{{Statement of Results}}
In this section, we first review some of the basic concepts
on RDS, which are taken from Arnold \cite{A}. Then we state our main results.

\subsection{Random Dynamical Systems}

Let $(\Omega, \mathcal{F}, {P})$ be a
probability space and $\mathbb{T}$ denote one of the sets: $\mathbb{R}$, $\mathbb{R}^+$, $\mathbb{Z}$, and $\mathbb{Z}^+$. $\mathbb{T}$ is endowed with its Borel $\sigma$-algebra $\mathcal{B}(\mathbb{T})$.
Let $\theta=(\theta^t)_{t\in \mathbb{R}}$ be a
measurable ${P}$-measure preserving flow on
$\Omega$, see Arnold \cite{A}.  $(\Omega, \mathcal{F}, {P}, \theta^t)$  is called a metric
dynamical system over the probability space $(\Omega, \mathcal{F}, {P})$. This metric dynamical system models the evolution of noise of the system. For the discrete time metric dynamical system, we replace $\mathbb{R}$ by $\mathbb{Z}$.
\smallskip

As an example, we take $(\Omega, \mathcal{F}, {P})$ to be  the Wiener space, i.e.,  $\Omega =\{
\omega\,:\, \omega(\cdot) \in C(\mathbb{R}, U),
\omega(0)=0\}$ for some separable Hilbert space $U$ endowed with the  compact open topology, $\mathcal{F}=\mathcal{B}(\Omega)$ and
${P}$ is the Wiener measure for a trace class covariance operator $Q$ on $U$. In fact, $\Omega$ is a Polish space.
We define a measurable flow $\theta^t$ on the probability space
$(\Omega, \mathcal{F}, \mathbb{P})$ by the Wiener shift
$(\theta^t\omega)(\cdot)=\omega(t+\cdot)-\omega(t)$ for $t\in\mathbb{R}$.  It is well-known that ${P}$ is invariant and ergodic under
$\theta^t$.


\smallskip

A random dynamical system in a metric space $X$  over the metric dynamical system
$(\Omega, \mathcal{F}, {P}, \theta^t)$   is a measurable map
$$
\phi: \mathbb{T} \times \Omega \times X \to X, \quad (t, \omega,
x) \mapsto \phi(t, \omega, x),
$$
and $\phi(t, \omega)(x):=\phi(t, \omega, x)$ forms a cocycle:
\[
\phi(0, \omega)=Id, \quad \hbox{ for all }\; \omega \in \Omega,
\]
\[
\phi(t+s,\omega)=\phi(t,\theta^{s}\omega)\circ \phi(s,\omega),
\quad \hbox{ for all }\; s, t \in \mathbb{T}, \quad\omega \in
\Omega.
\]
When $\phi(t, \omega, \cdot):X \to X$ is
continuous, $\phi$ is called a continuous random dynamical system.

\smallskip
As we mentioned in the introduction, a typical example in finite dimensional space  is the solution operator of stochastic differential equations.

\smallskip
 An infinite dimensional random dynamical system can be generated by the solutions of partial differential equations driven by a stochastic process of the form
\[
u_t = \Delta u+ f(u, x, \theta_t \omega), \quad x\in U
\] with the Dirichlet boundary condition or the Neumann boundary,  where $U \subset \mathbb{R}^n$ is a bounded region with a smooth boundary.    It can also be generated by the solutions of stochastic partial differential equations of the form
\begin{equation*}
    du =(Au+F(u)) dt + d{W},
\end{equation*}
where $A$ is an elliptic operator, $F$ is a smooth nonlinear functional, and $d{W}$ is a white noise given as the generalized temporal
differential of a Wiener process with continuous paths in the phase space.

\smallskip
Other examples of random dynamical systems are generated by the solutions of stochastic differential equations driven by a fractional Brownian motion \cite{GLS}.

\subsection{Random Invariant Sets and Random Attractors}

We first recall that a multifunction
$M=(M(\omega))_{\omega\in\Omega}$ of nonempty closed sets
$M(\omega),\,\omega\in\Omega$, contained in $X$ is
called {\it a random set} if
\[
\omega\mapsto\inf_{y\in M(\omega)}d(x,y)
\]
is a random variable for any $x\in X$.  A random set $M$ is  invariant for random dynamical system $\phi$ if
\[
\phi(t,\omega,M(\omega))\subset M(\theta_t\omega)\quad \text{for} \quad t\geq 0.
\]
A random set $\mathcal{A}=\{\mathcal{A}(\omega)\}_{\omega \in \Omega}$
of $X$
is called {\it a global random   attractor}  for
  $\phi$
if the following  conditions are satisfied, for $\mathbb{P}$-a.e. $\omega \in \Omega$,
\begin{itemize}
\item[(i)]  $\mathcal{A}(\omega)$ is compact;

\item[(ii)] $\{\mathcal{A}(\omega)\}_{\omega \in \Omega}$ satisfies for $t\geq 0$:
\[
\phi(t, \omega, \mathcal{A}(\omega))=\mathcal{A}(\theta_t\omega);
\]

\item[(iii)] $\{\mathcal{A}(\omega)\}_{\omega \in \Omega}$
attracts  every tempered random set
 $B = \{B(\omega)\}_{\omega \in \Omega}$, that is,
$$ \lim_{t \to  \infty} d (\phi(t, \theta_{-t}\omega, B(\theta_{-t}\omega)), \mathcal{A}(\omega))=0,
$$
where $d$ is the Hausdorff semi-metric given by
$d(Y,Z) =
  \sup_{y \in Y }
\inf_{z\in  Z}  \| y-z\|_{X}
 $ for any $Y\subseteq X$ and $Z \subseteq X$.
\end{itemize}

The study of global random attractors was initiated by Ruelle \cite{Ruelle1}. The fundamental theory of global random attractors for
stochastic partial differential equations was developed by Crauel,
Debussche, and Flandoli \cite{CDF}, Crauel and Flandoli
\cite{cr-fl}, Flandoli and Schmalfuss \cite{fl-schm}, Imkeller and
Schmalfuss \cite{ imk-schm},  and others. It has recently attracted more attention due to new equations and models arising in applications such as stochastic infinite dimensional lattice dynamical systems \cite{BLL}.

Due to the unbounded fluctuations in the systems caused by the white noise, the concept
of pull-back global random attractor was introduced to capture the essential dynamics with possibly extremely wide fluctuations. This
is significantly different from the deterministic case.

\subsection{Main Result}
In this paper, we study the complicated dynamics of infinite dimensional random dynamical systems restricted to random invariant sets such as global attractors.  We assume that $(\Omega, \mathcal{F},P, \theta)$ is a Polish system (see Section 3) and  $X$ is a Polish space with the distance function $d$. We consider a continuous random dynamical system $\phi(n, \omega, x)$ generated by a random map $\phi(\omega)(x)$ defined on $X$ over the metric dynamical system $(\Omega, \mathcal{F},P, \theta)$, i.e.,
\[
\phi(n, \omega, x)=
\begin{cases} \phi(\theta^{n-1}\omega)\cdots
\phi(\omega)(x), & n >0, \\
I, & n=0. \\
\end{cases}
\] Here, $\phi(\omega):X\rightarrow X$ is
continuous almost surely.

\smallskip
We consider a $\phi$-invariant random set ${K}\in {\mathcal F}\otimes\mathcal B(X)$  with compact $\omega$-section ${K}(\omega)$.
For $\omega \in \Omega$, $\epsilon>0$  and $n\ge 1$, we define $$d^\omega_n(x,y)=\max \limits_{0\le k\le n-1}
d(\phi(k, \omega, x),\phi(k, \omega, y)), \, \text{ for }x,y\in X.$$  A
subset $E$ of $K(\omega)$ is called {\it $(\omega,n,\epsilon,\phi)$-separated subset of  $K(\omega)$} if for all
$x,y\in E,x\neq y$, one has $d^\omega_n(x,y)>\epsilon$. We denote by $r_n(K,\omega,\epsilon,\phi)$ the maximal
cardinality of all $(\omega,n,\epsilon,\phi)$-separated subset of $K_\omega$.  The {\it topological entropy}  of $(\phi,K)$ based on Bowen and Dinaburg's definition is given by
$$h_{\text{top}}(\phi,K):=\lim_{\epsilon\rightarrow 0} \limsup_{n\rightarrow +\infty} \frac{1}{n} \int_\Omega
\log r_n(K,\omega,\epsilon,\phi)  d P(\omega).$$
See,   Bogenschutz
\cite{B92} and Kifer \cite{K01}  for  related notions  in the case of $X$ being compact.

\smallskip

 Let $\phi$ be an injective continuous random dynamical system
on Polish space $X$   over an ergodic Polish system $(\Omega,
\mathcal{F}, {P},\theta)$. Let $K$ be a $\phi$-invariant random set
with compact $\omega$-section $K(\omega)$ such that
$\phi(\omega)(K(\omega))=K(\theta \omega)$. Subsets $U_1,U_2$ of $X$
is called {\it weak Horseshoe} of $(\phi,K)$, if the following
properties hold
\begin{enumerate}
\item $U_1$ and $U_2$ are non-empty, closed, and bounded subsets of $X$ and
$d(U_1,U_2)>0$.

\item there is a $b>0$ satisfying for $P$-a.e. $\omega\in \Omega$,
there exists $M_{b,\omega}\in \mathbb{N}$ such that for any natural
number $m\ge M_{b,\omega}$, there is a subset $J_m\subset
\{0,1,2,\cdots,m-1\}$ with $|J_m|\ge bm$(positive density), and for any $s\in
\{1,2\}^{J_m}$, there exists $x_s\in K(\omega)$ with
$\phi(j,\omega,x_s)\in U_{s(j)}\cap K(\theta^j \omega)$ for any
$j\in J_m$.
\end{enumerate}

The first result is on the existence of a weak horseshoe.
\begin{thm}\label{MTH1} Let $\phi$ be an injective continuous random dynamical system on Polish space $X$   over an ergodic Polish system
$(\Omega, \mathcal{F}, {P},\theta)$. Let $K$ be a $\phi$-invariant
random set with compact $\omega$-section $K(\omega)$ such that
$\phi(\omega)(K(\omega))=K(\theta \omega)$. If
$h_{\text{top}}(\phi,K)>0$, then there exists a weak Horseshoe
$\{U_1,U_2\}$ for $(\phi,K)$.
\end{thm}

Let $(\Omega, \mathcal{F}, {P},\theta)$ be a {\it compact metric system}. Namely,
$\Omega$ is a compact metric space, $\mathcal{F}$ is the Borel $\sigma$-algebra $\mathcal{B}_{\Omega}$ of $\Omega$,
$P$ is a Borel probability measure
on $\Omega$ and $\theta:\Omega\rightarrow \Omega$ is a continuous map preserving the measure $P$.

A multifunction
$M=(M(\omega))_{\omega\in\Omega}$ of nonempty closed sets
$M(\omega),\,\omega\in\Omega$, contained in $X$ is
called {\it a strongly compact random set} if the following  conditions are satisfied, for each $\omega \in \Omega$,
\begin{itemize}
\item[(i)]  $M(\omega)$ is compact.

\item[(ii)] the function
$\omega\mapsto\inf_{y\in M(\omega)}d(x,y)$
is lower semi-continuous  for any $x\in X$.
\end{itemize}
It is not hard to see that if a multifunction
$M=(M(\omega))_{\omega\in\Omega}$ of nonempty closed sets
$M(\omega),\,\omega\in\Omega$, contained in $X$ such that the set $\bigcup_{\omega\in \Omega}\{\omega\}\times M(\omega)$ is a compact subset of $\Omega\times X$, then
$M$ is strongly compact random set.

By a {\it full horseshoe of two symbols} we mean that there exist  subsets $U_1,U_2$ of $X$  such that the following
properties hold
\begin{enumerate}
\item $U_1$ and $U_2$ are non-empty, closed, and bounded subsets of $X$  and
$d(U_1,U_2)>0$.

\item  there is  a constant $b>0$ satisfying for $P$-a.e. $\omega\in \Omega$,
there exists $J(\omega)\subset  \mathbb{N}_0$ such that the limit
$\lim_{m\rightarrow +\infty}\frac{1}{m}|J(\omega)\cap \{0,1,2,\cdots, m-1\}|$ exists and is larger than or equal to $b$
(positive density), and for any $s\in
\{1,2\}^{J(\omega)}$, there exists $x_s\in K(\omega)$ with
$\phi(j,\omega,x_s)\in U_{s(j)}\cap K(\theta^j \omega)$ for any
$j\in J(\omega)$.
\end{enumerate}

The second result is the existence of a full horseshoe.

\begin{thm}\label{MTH3} Let $\phi$ be an injective continuous random dynamical system on Polish space $X$   over an ergodic compact metric system
$(\Omega, \mathcal{F}, {P},\theta)$ satisfying the map $(\omega,x)\mapsto \phi(\omega)x$ is a continuous map from $\Omega\times X$ to $X$. Let $K$ be a $\phi$-invariant
strongly compact random set such that
$\phi(\omega)(K(\omega))=K(\theta \omega)$. If
$h_{\text{top}}(\phi,K)>0$, then there exists a full Horseshoe
$\{U_1,U_2\}$ for $(\phi,K)$.
\end{thm}
As a consequence, we have
\begin{cor}\label{Cor1} Let $\phi$ be an injective continuous on Polish space $X$   and $K$ be a $\phi$-invariant
compact set. If
$h_{\text{top}}(\phi,K)>0$, then there exists a full Horseshoe
$\{U_1,U_2\}$ for $(\phi,K)$.
\end{cor}

The notion of Li-Yorke chaos was introduced in \cite{LY} for interval maps. With a small modification this notion can be extended to a metric space. Following the idea of Li and Yorke, a subset $D$ of $K({\omega})$,  is called {\it $\kappa$-chaotic set} for $(\omega,\phi)$, where $\omega\in
\Omega$   and $\kappa>0$, if for every pair $(x_1,x_2)$ of distinct
points in $D$, one has $$\liminf_{n\rightarrow +\infty} d(\phi(n, \omega, x_1),\phi(n, \omega, x_2))=0 \text{ and }
\limsup_{n\rightarrow +\infty} d(\phi(n, \omega, x_1),\phi(n, \omega, x_2))\ge \kappa.$$


\smallskip
The final result is about the positive entropy implying the existence of Li-Yorke chaos.

\begin{thm}\label{MTH} Let $\phi$ be an injective continuous random dynamical system on Polish space $X$   over an ergodic Polish system
$(\Omega, \mathcal{F}, {P},\theta)$. Let $K$ be a $\phi$-invariant random set with compact $\omega$-section $K(\omega)$ such that $\phi(\omega)(K(\omega))=K(\theta \omega)$. If $h_{\text{top}}(\phi,K)>0$,
then there exists  $\kappa>0$ such that for ${P}$-a.e. $\omega\in \Omega$ there is a
$\kappa$-chaotic set $S(\omega)\subset K(\omega)$ of a union of countably many Cantor sets for $(\omega,\phi)$. \end{thm}

\section{Basic Concepts and Lemmas on Measurable Dynamical Systems}
In this section, we review some of basic concepts and results from the theory of measurable dynamical systems and introduce several lemmas that we need for the proofs of the main theorems.

\subsection{Various Dynamical Systems} In this paper for a probability space $(X,\mathcal{B},\mu)$ we always require
that $\mathcal{B}$ is countably generated ($\mu$-mod $0$), that is, there exists $\{ A_i\}_{i=1}^\infty \subset
\mathcal{B}$ such that for any $A\in \mathcal{B}$ and  $\epsilon>0$ there is $i:=i(A,\epsilon)\in
\mathbb{N}$ satisfying $\mu(A\Delta A_i)<\epsilon$. {\it A measure-theoretic dynamical system} (MDS)
$(X,\mathcal{B},\mu,T)$  is a measure-preserving map $T$ on a probability space $(X,\mathcal{B},\mu)$.

{\it A Polish probability space}  $(X,\mathcal{B}_X,\mu)$ means that $X$ is a separable topological space
whose topology is metrizable by a complete metric, $\mathcal{B}_X$ is the Borel $\sigma$-algebra, and $\mu$
is a  Borel probability measure on $X$. {\it A Polish system}  $(X,\mathcal{B}_X,\mu,T)$  is a
measure-preserving map $T$ on a Polish space $(X,\mathcal{B}_X,\mu)$. {\it A Lebesgue system}
$(X,\mathcal{B},\mu,T)$  is a measure-preserving map $T$ on a Lebesgue space $(X,\mathcal{B},\mu)$ (see
\cite{R}). For a Polish system
$(X,\mathcal{B}_X,\mu,T)$, the MDS $(X,\mathcal{B}_\mu,\mu,T)$  constitutes  a Lebesgue system, where
$\mathcal{B}_\mu$ is the completion of the Borel $\sigma$-algebra $\mathcal{B}_X$ with  respect to $\mu$.

A MDS $(Y,\mathcal{D},\nu,S)$  is  said to be {\it a factor} of $(X,\mathcal{B},\mu,T)$ if there is  a
measure-preserving map $\pi:(X,\mathcal{B},\mu)\rightarrow (Y,\mathcal{D},\nu)$ such that $\pi T=S\pi$.
Equivalently, we say that $(X,\mathcal{B},\mu,T)$ is an extension of $(Y,\mathcal{D},\nu,T)$. In this case,
we also  say $\pi:(X,\mathcal{B},\mu,T)\rightarrow (Y,\mathcal{D},\nu,S)$ is  {\it a factor map}.

\subsection{Conditional entropy} In this subsection, we first recall the notation of the conditional entropy
of a MDS. Then we state some results about the conditional entropy. Consider an MDS $(X,\mathcal{B},\mu,T)$. {\it A
partition} of $X$ is a family of pairwise disjoint sets in $\mathcal{B}$ whose union is $X$. For a given
partition $\alpha$ of $X$ and $x\in X$, we denote by $\alpha(x)$ the atom of $\alpha$ containing $x$.

We denote the set of finite partitions of $X$ by $\mathcal{P}_X(\mathcal{B})$ or for simplicity
$\mathcal{P}_X$. Given two partitions $\alpha,\beta$ of $X$, $\alpha$ is said to be finer than $\beta$
(denote by $\alpha\succeq \beta$) if each element of $\alpha$ is  contained in some element of $\beta$. Let
$\alpha\vee\beta=\{A\cap B: A\in \alpha,B\in \beta \}$.

For any given $\alpha \in \mathcal{P}_X$ and any sub-$\sigma$-algebra $\mathcal{C}$ of $\mathcal{B}$, let
$$H_{\mu}(\alpha)=\sum_{A\in \alpha} -\mu(A) \log \mu(A)\ \text{and}\ \
H_{\mu}(\alpha|\mathcal{C})=\sum_{A\in \alpha} \int_X
-\mathbb{E}_\mu(1_A|\mathcal{C}) \log
\mathbb{E}_\mu(1_A|\mathcal{C}) d \mu,$$ where
$\mathbb{E}_\mu(1_A|\mathcal{C})$ is the conditional expectation of
the characterization function $1_A$ of $A$ with respect to
$\mathcal{C}$. One standard fact states that $H_\mu(\alpha|
\mathcal{C})$ increases with respect to $\alpha$ and decreases with
respect to $\mathcal{C}$. When $T^{-1}\mathcal{C}\subseteq
\mathcal{C}$, it is not hard to see that
$H_\mu(\bigvee_{i=0}^{n-1}T^{-i}\alpha|\mathcal{C})$ is non-negative
and sub-additive sequence for a given $\alpha\in \mathcal{P}_X$, so
we can define $$h_\mu(T,\alpha|\mathcal{C})=\lim_{n\rightarrow
+\infty} \frac{1}{n}
H_\mu(\bigvee_0^{n-1}T^{-i}\alpha|\mathcal{C})=\inf_{n\ge 1}
\frac{1}{n} H_\mu(\bigvee_0^{n-1}T^{-i}\alpha|\mathcal{C}).$$ The
{\it measure-theoretic entropy} of $\mu$ with respect to
$\mathcal{C}$ is defined as $$h_\mu(T|\mathcal{C})=\sup_{\alpha \in
\mathcal{P}_X} h_\mu(T,\alpha|\mathcal{C}). $$

Let $\pi:(X,\mathcal{B},\mu,T)\rightarrow (Y,\mathcal{D},\nu,S)$ be a factor map between two MDSes. We define
{\it the conditional entropy}  of $\mu$ with respect to $\pi$ as
$$h_\mu(T|\pi)=h_\mu(T|\pi^{-1}(\mathcal{D})).$$

The following result is a generalization of Abramov-Rohlin formula borrowed from \cite{BC}.
\begin{lem} \label{GAR} (Generalized Abramov-Rohlin formula) Let
$\pi:(X,\mathcal{B},\mu,T)\rightarrow (Y,\mathcal{D},\nu,S)$ and $\psi:(Y,\mathcal{B},\nu,S)\rightarrow
(Z,\mathcal{A},\lambda,R)$ be two factor maps between two MDSes. Then  $\psi\circ
\pi:(X,\mathcal{B},\mu,T)\rightarrow (Z,\mathcal{A},\lambda,R)$ is also a factor map between MDSes and
$$h_\mu(T|\psi\circ \pi)=h_\mu(T|\pi)+h_\nu(S|\psi).$$ \end{lem}

\smallskip
For a given Lebesgue system $(X,\mathcal{B},\mu,T)$, let
$$\bar{X}=\{ \bar{x}=(x_i)_{i\in  \mathbb{Z}}\in
X^{\mathbb{Z}}: \, Tx_{i}=x_{i+1}, i\in \mathbb{Z}\}
$$ and
$(\bar{X},\bar{\mathcal{B}},\bar{\mu},\bar{T})$ be the natural extension of $(X,\mathcal{B},\mu,T)$ (see
\cite[Section 3.7]{R}). More precisely, for $n\in \mathbb{Z}$, let $\Pi_{n,X}:\bar{X}\rightarrow X$ with
$\Pi_{n,X}((x_i)_{i\in \mathbb{Z}})=x_n$ for $\bar{x}=(x_i)_{i\in \mathbb{Z}}\in X^{\mathbb{Z}}$.
Set
$$\bar{\mathcal{B}}_n=\Pi_{n,X}^{-1}(\mathcal{B}).$$ Clearly, $\bar{\mathcal{B}}_i\supseteq
\bar{\mathcal{B}}_{i+1}$ for each $i\in \mathbb{Z}$. Let $\mathcal{D}_{\bar{X}}=\bigcup_{i\in
\mathbb{Z}}\bar{\mathcal{B}}_i$. Then $\mathcal{D}_{\bar{X}}$ is a algebra of subsets of  $\bar X$. The Lebesgue
measure $\bar{\mu}$ on  $\mathcal{D}_{\bar{X}}$ satisfies $\bar{\mu}(\Pi_{n,X}^{-1}(A))=\mu(A)$  for $A\in
\mathcal{B}$ and  $n\in \mathbb{Z}$.  $\bar{\mathcal{B}}$ is the completion of the $\sigma$-algebra
generated by $\mathcal{D}_{\bar{X}}$ with respect to $\bar{\mu}$. The self-map $\bar{T}$  defined on
$\bar{X}$ by
 $$\bar{T}( (x_i)_{i\in  \mathbb{Z}})=(Tx_{i})_{i\in \mathbb{Z}}=(x_{i+1})_{i\in \mathbb{Z}}$$
is an invertible measure-preserving map on Lebesgue space $(\bar{X},\bar{\mathcal{B}},\bar{\mu})$. Thus,
$(\bar{X},\bar{\mathcal{B}},\bar{\mu},\bar{T})$ is an invertible Lebesgue system.

Let $\Pi_X:=\Pi_{0,X}$.  Then $\Pi_X:(\bar{X},\bar{\mathcal{B}},\bar{\mu},\bar{T})\rightarrow
(X,\mathcal{B},\mu,T)$ is a factor  map,  which is called {\it the natural extension} of
$(X,\mathcal{B},\mu,T)$. In  \cite{R61} it is proved that $(\bar{X},\bar{\mathcal{B}},\bar{\mu},\bar{T})$ is
ergodic if and only if $(X,\mathcal{B},\mu,T)$ is ergodic.

The next lemma is on the entropy of the extended map conditional to the natural extension. Its proof is given in Appendix A.
\begin{lem} \label{na=zero} Let $\Pi_X:(\bar{X},\bar{\mathcal{B}},\bar{\mu},\bar{T})\rightarrow
(X,\mathcal{B},\mu,T)$ be  the natural extension of Lebesgue system $(X,\mathcal{B},\mu,T)$. Then
$h_{\bar{\mu}}(\bar{T}|\Pi_X)=0$.
\end{lem}

\subsection{Relative Pinsker $\sigma$-algebra} In this
subsection, we recall some notations and results on the  relative
Pinsker $\sigma$-algebra.

Let $(X,\mathcal{B},\mu,T)$ be an invertible Lebesgue system.
A $T$-invariant sub-$\sigma$-algebra $\mathcal{F}$ (i.e., $T^{-1}\mathcal{F}=\mathcal{F}$) of $\mathcal{B}$
determines an invertible Lebesgue  factor $(Y,\mathcal{D},\nu,S)$ of $(X,\mathcal{B},\mu,T)$, that is, there
exists a factor map $\pi:(X,\mathcal{B},\mu,T)\rightarrow (Y,\mathcal{D},\nu,S)$ between two invertible
Lebesgue systems such that $\pi^{-1}(\mathcal{D})=\mathcal{F}$. This factor is  unique, up to isomorphism
(see for example \cite[Section 4.1]{P69}).

For a  factor map $\pi:(X,\mathcal{B},\mu,T)\rightarrow (Y,\mathcal{D},\nu,S)$  between two invertible
Lebesgue systems, there is a set of conditional probability measures $\{\mu_y\}_{y\in Y}$ with the following
properties: \begin{itemize} \item $\mu_y$  is a Lebesgue  measure on $X$ with $\mu_y(\pi^{-1}(y))=1$ for all
$y\in Y$.

\item for each $f \in L^1(X,\mu)$,  one  has $f \in L^1(X,\mu_y)$ for $\nu$-a.e. $y\in Y$, the map $y \mapsto
    \int_X f\,d\mu_y$ is in $L^1(Y,\nu)$ and $\int_Y \left(\int_X f\,d\mu_y \right)\, d\nu(y)=\int_X f
    \,d\mu$. \end{itemize}

\medskip In this case,  we say that $\mu=\int_Y \mu_y d \nu(y)$ is {\it disintegration of $\mu$ relative to
the  factor $(Y,\mathcal{D},\nu,S)$}. The measures $\{ \mu_y\}$ are essentially unique; that is, $\{ \mu_y\}$
and $\{  \mu_y'\}$  have the above properties, then $\mu_y=\mu_y'$  for $\nu$-a.e. $y\in Y$. In particular,
it follows that $T\mu_y=\mu_{S y}$  for $\nu$-a.e. $y\in Y$. The conditional  expectations and the conditional measures
are related by
\begin{equation} \label{meas3} \mathbb{E}_{\mu}(f|\pi^{-1}\mathcal{D})(x)=\int_X f\,d\mu_{\pi(x)} \ \ \text{for
$\mu$-a.e. } x\in X
\end{equation} for every  $f\in
L^1(X,\mathcal{B},\mu)$.

The product of $(X,\mathcal{B},\mu,T)$ with itself relative to  factor $(Y,\mathcal{D},\nu,S)$ is the MDS
$$(X\times X,\mathcal{B}\times \mathcal{B},\mu \times_Y\mu, T\times T),$$ where the measure $$(\mu\times_Y
\mu)(B)=\int_{Y}(\mu_y\times \mu_y)(B) \, d \nu(y), \ \forall B\in \mathcal{B}\times \mathcal{B}.$$ The
measure $\mu\times_Y \mu$ is   $T\times T$-invariant and  is supported on $$R_\pi:=\{(x_1,x_2)\in X\times
X:\pi(x_1)=\pi(x_2)\}$$ since each measure $\mu_y\times \mu_y$  is supported
on $\pi^{-1}(y)\times \pi^{-1}(y)$.

In  the following,  we are to introduce  the relative Pinsker $\sigma$-algebra for a  factor map between two
invertible Lebesgue systems. First, we recall the relative Pinsker formula from \cite[p.66]{P81}.

\begin{lem}\label{RPF} (Relative Pinsker formula) Let $(X,\mathcal{B},\mu,T)$ be an invertible Lebesgue
systems and $\mathcal{A}$ be a sub-$\sigma$ algebra of $\mathcal{B}$ with $T^{-1}\mathcal{A}=\mathcal{A}$. Then
for any $\alpha,\beta\in \mathcal{P}_X$, $$h_\mu(T,\alpha\vee
\beta|\mathcal{A})=h_\mu(T,\beta|\mathcal{A})+h_\mu(T,\alpha|\beta^T\vee \mathcal{A})$$ where
$\beta^T=\bigvee_{i\in  \mathbb{Z}}T^{-i}\beta$. \end{lem}

Let $\pi:(X,\mathcal{B},\mu,T)\rightarrow (Z,\mathcal{A},\lambda,R)$ be a  factor map between two invertible
Lebesgue systems. Put
$$P_\mu(\pi)=\{ A\in \mathcal{B}: h_\mu(T,\{  A,X\setminus
A\}|\pi^{-1}(\mathcal{A}))=0\}.
$$ It is well known that from Lemma \ref{RPF} it follows that  $P_\mu(\pi)$ is the smallest
sub-$\sigma$-algebra of $\mathcal{B}$ containing $\{ \alpha\in \mathcal{P}_X: h_\mu(T,\alpha|\pi^{-1}(\mathcal{A}))=0\}$ and $T^{-1}(P_\mu(\pi))=P_\mu(\pi)$.  $P_\mu(\pi)$ is called the {\it
relative Pinsker  $\sigma$-algebra} of $(X,\mathcal{B},\mu,T)$ with respect  to $\pi$. Note that
$$\pi^{-1}(\mathcal{A})\subseteq P_\mu(\pi)\subseteq \mathcal{B}.
$$ There exist an invertible Lebesgue system
$(Y,\mathcal{D},\nu,S)$  and  two factor maps $$\pi_1:(X,\mathcal{B},\mu,T)\rightarrow (Y,\mathcal{D},\nu
,S),\ \ \pi_2:(Y,\mathcal{D},\nu,S)\rightarrow (Z,\mathcal{A},\lambda,R)$$ such that $\pi_2\circ \pi_1=\pi$
and $\pi_1^{-1}(\mathcal{D})=P_\mu(\pi)$ (see for example \cite{P69}). The factor map $\pi_1:(X,\mathcal{B},\mu,T)\rightarrow
(Y,\mathcal{D},\nu ,S)$ is called {\it the Pinsker factor  of $(X,\mathcal{B},\mu,T)$ with  respect to
$\pi$}.

The  following result is  well known (see for example \cite[Theorem 2.1 and Theorem 2.3]{BGKM} or \cite[Lemma
4.1]{ZGH}).
\begin{lem} \label{key-lem} Let $\pi:(X,\mathcal{B},\mu,T)\rightarrow (Z,\mathcal{A},\lambda,R)$
be a  factor map between two invertible Lebesgue systems. Let  $\pi_1:(X,\mathcal{B},\mu,T)\rightarrow
(Y,\mathcal{D},\nu ,S)$ be the Pinsker factor  of $(X,\mathcal{B},\mu,T)$ with  respect to
$(Z,\mathcal{A},\lambda,R)$ and  $\mu=\int_Y  \mu_y d \nu$ be the disintegration of $\mu$ relative to the
factor $(Y,\mathcal{D},\nu,S)$. If $(X,\mathcal{B},\mu,T)$  is ergodic and $h_\mu(T|\pi)>0$, then
\begin{enumerate} \item $\mu_y$ is non-atomic (that is $\mu_y(\{x\})=0$ for each $x\in  X$) for $\nu$-a.e.
$y\in Y$.

\item $(X\times X,\mathcal{B}\times \mathcal{B},\mu \times_Y\mu, T\times T)$ is ergodic. \end{enumerate}
 \end{lem}


The following result is also well known (see for example \cite[Lemma 3.3]{HYZ}).
\begin{lem} \label{key-lem-1} Let $\pi:(X,\mathcal{B},\mu,T)\rightarrow (Z,\mathcal{A},\lambda,R)$
be a  factor map between two invertible Lebesgue systems. Let
$P_{\mu}(\pi)$ be the relative Pinsker $\sigma$-algebra of
$(X,\mathcal{B},\mu,T)$ relative to $\pi$. If $\alpha \in
\mathcal{P}_X$, then
$$\lim_{m \rightarrow +\infty} h_\mu(T^m,\alpha|\pi^{-1}(\mathcal{A}))=H_\mu(\alpha|P_\mu(\pi)).$$
 \end{lem}

\subsection{Entropy of compact Random set}

In this subsection, we first introduce the entropy of compact random set for the random dynamical system. We then present the variational principle.

Let $\phi(n, \omega, x)$ a random dynamical system (RDS) in a Polish space $X$ over the metric dynamical system  $(\Omega, \mathcal{F},P,\theta)$. We assume that $\Omega$ is a Polish system. Note that $\phi(n, \omega, x)$ is generated by a random map $\phi(\omega)(x):=\phi(1, \omega, x)$. We assume $\phi(\omega):X\rightarrow X$ is
continuous for $P$-a.e. $\omega\in \Omega$.

\begin{de} Suppose that $\phi$ is a RDS on $X$. The map
$$\Phi:\Omega\times X\rightarrow \Omega\times X,\ \ (\omega,x)\rightarrow (\theta \omega,\phi(\omega)x)$$ is said
to be {\it skew product} induced by $\phi$. \end{de}

Let $\mathcal{F}\times \mathcal{B}_X$ be the smallest
$\sigma$-algebra  on $\Omega\times X$ with respect to which both the canonical projections
$\Pi_X:\Omega\times X\rightarrow X$ and $\pi_\Omega:\Omega\times X\rightarrow \Omega$ are measurable. A
probability measure $\mu$ on the measurable space $(\Omega\times X,\mathcal{F}\times  \mathcal{B}_X)$ is said
to have {\it marginal $P$} on $\Omega$ if $\mu\circ \pi_\Omega^{-1}=P$.  Denote by
$\mathcal{P}_P(\Omega\times X)$ the collection of such measures. Let $\mathcal{M}_P(\Omega\times X,f)$ denote
the set of $\Phi$-invariant elements of $\mathcal{P}_P(\Omega\times X)$. If $(\Omega, \mathcal{F},P,\theta)$ is
an ergodic MDS, then we may consider the set $\mathcal{E}_P(\Omega\times X)$ of ergodic elements in
$\mathcal{M}_P(\Omega\times X,f)$.

Let $K\in \mathcal{F}\times \mathcal{B}_X$ be a forward $\phi$-invariant random set with compact $\omega$-section $K(\omega)$.  Then, there  exists $\mu\in  \mathcal{M}_P(\Omega\times X)$
that is supported on $K$, i.e., $\mu(K)=1$ (see \cite{C} and \cite[Theorem 1.6.13]{A}). For simplicity, we
denote $$\mathcal{M}^K_P(\Omega\times X)=\{ \mu\in \mathcal{M}_P(\Omega\times X):\mu(K)=1\}.$$

Let $\mu\in \mathcal{M}^K_P(\Omega\times X)$. Then, there is a decomposition $d\mu(\omega,x)=d\mu_\omega(x) dP(\omega)$ of $\mu$ into its sample measures $\mu_\omega,\omega\in \Omega$  and $P$ (see \cite{A} and \cite[Section
3]{C}). $\mu_\omega(K_\omega)=1$ for $P$-a.e. $\omega\in \Omega$. Since $K$ is a Borel subset of
$\Omega\times X$, $K$ is also Polish space and the Borel $\sigma$-algebra $\mathcal{B}_K$ of $K$ is just $\{
A\cap K: A\in \mathcal{F}\times \mathcal{B}_X\}$. Thus $(K,\mathcal{B}_K,\mu, \Phi)$ is a Polish system and
$\pi_\Omega: (K,\mathcal{B}_K,\mu, \Phi)\rightarrow (\Omega, \mathcal{F},P,\theta)$ is a factor map bewteen two
Polish systems.

 {\it The entropy of RDS $\phi$} with respect to $\mu$ is defined by
$$h_\mu(\phi):=h_\mu(\Phi|\pi_\Omega)=h_\mu(\Phi|\pi_\Omega^{-1}(\mathcal{F})).$$ That is, the entropy of RDS
$\phi$ with respect to $\mu$ is the entropy of $(K,\mathcal{B}_K,\mu, \Phi)$ with respect to
$\pi_\Omega^{-1}(\mathcal{F})$. By Bogenschutz \cite{B},
$$h_\mu(\phi)=\sup\{
h_\mu(\Phi,\pi^{-1}_X\alpha|\pi_\Omega^{-1}(\mathcal{F})):\alpha \text{ is a finite Borel partition of
}X\}.$$

If in addition that $P$ is ergodic, then $$\Gamma=\{ \mu\in \mathcal{P}_P(\Omega\times X):\mu(K)=1,\  \mu \text{
invariant for }\Phi\}$$ is a compact (in the narrow topology of $\mathcal{P}_P(\Omega\times X)$, which is a
metrizable topology) and convex, and its extremal points are ergodic by Lemma 6.19 in \cite{C}, particularly
there exists $\mu\in  \mathcal{E}_P(\Omega\times X)$ that is supported on $K$. We use
$\mathcal{E}^K_P(\Omega\times X)$ to denote the set of the ergodic elements of $ \mathcal{M}^K_P(\Omega\times X)$. Then
$\mathcal{E}^K_P(\Omega\times X)$ is the set of extremal points of $  \mathcal{M}^K_P(\Omega\times X)$, which is a Borel subset of $ \mathcal{M}^K_P(\Omega\times X)$.


\medskip

Using the approach of Kifer \cite{K01} and Lemma \ref{ergodic}, one has the following result.

\begin{prop} \label{VP} (Variational principle) Let  $\phi$ be a continuous RDS on Polish space $X$ with Borel
$\sigma$-algebra $\mathcal{B}_X$ over Polish system $(\Omega, \mathcal{F},P,\theta)$. Let $K$ be a
forward $\phi$-invariant random set with compact $\omega$-section $K(\omega)$. Then $$h_{\text{top}}(\phi,K)=\sup \{h_\mu(\phi):\mu\in
\mathcal{M}_P^K(\Omega\times X,\phi)\}.$$ If in addition $(\Omega, \mathcal{F},P,\theta)$ is ergodic, then
$$h_{\text{top}}(\phi,K)=\sup \{h_\mu(\phi):\mu\in \mathcal{E}_P^K(\Omega\times X,\phi)\}.$$ \end{prop}

\section{Proof of Theorem \ref{MTH1}}
In this section, we first introduce a combinatorial lemma and some results on relative Pinsker $\sigma$-algebra and conditional entropy of a finite measurable partition.
We then prove  Theorem \ref{MTH1}.

\subsection{Condition Entropy and a combinatorial lemma}

Let $\pi:(X,\mathcal{B},\mu,T)\rightarrow (Z,\mathcal{A},\lambda,R)$
be a  factor map between two invertible Lebesgue systems. Let
$\pi_1:(X,\mathcal{B},\mu,T)\rightarrow (Y,\mathcal{D},\nu ,S)$ be
the Pinsker factor  of $(X,\mathcal{B},\mu,T)$ with  respect to
$(Z,\mathcal{A},\lambda,R)$ and  $\mu=\int_Y  \mu_y d \nu$ be the
disintegration of $\mu$ relative to the factor
$(Y,\mathcal{D},\nu,S)$. It is well known that for $\nu$-a.e. $y\in
Y$, $\mu_y(\pi_1^{-1}(y))=1$.

Given $\ell\in \mathbb{N}$, let  $\alpha$ be a finite measurable
partition of $X$. Define a function
$$h_\mu(T^\ell,\alpha,y):=\lim_{n\rightarrow
+\infty}H_{\mu_y}(\alpha|\bigvee_{i=1}^n T^{-i\ell}\alpha).$$ Then
$h_\mu(T^\ell,\alpha,y)$ is a measurable function on $Y$ and
$h_\mu(T^\ell,\alpha,y)\le \log \#(\alpha)$. We have
\begin{lem}\label{lem-5.1} $\int_Y h_\mu(T^\ell, \alpha,y)d
\nu(y)=h_\mu(T^\ell,\alpha|\pi^{-1}(\mathcal{A}))$.
\end{lem}
\begin{proof} First, we  show that
 \begin{align}\label{peq-en}
 h_\mu(T^\ell,\alpha|\pi^{-1}(\mathcal{A}))=h_\mu(T^\ell,\alpha|P_\mu(\pi)).
\end{align}
Since $\pi^{-1}(\mathcal{A})\subseteq P_\mu(\pi)$, $h_\mu(T^\ell,\alpha|\pi^{-1}(\mathcal{A}))\ge h_\mu(T^\ell,\alpha|P_\mu(\pi))$.
Conversely, let $\beta_k\nearrow P_\mu(\pi)$ be  an increasing sequence of finite measurable partitions of $X$. Since $\beta_k\subseteq P_\mu(\pi)$,
$h_\mu(T,\beta_k|\pi^{-1}(\mathcal{A}))=0$ for $k\in \mathbb{N}$.
Moreover,
\begin{align*}
h_\mu(T^\ell,\beta_k|\pi^{-1}(\mathcal{A}))&\le
h_\mu(T^\ell,\bigvee_{j=0}^{\ell-1}T^{-j}\beta_k|\pi^{-1}(\mathcal{A}))
=\ell h_\mu(T,\beta_k|\pi^{-1}(\mathcal{A}))=0
\end{align*}
for $k\in \mathbb{N}$. Hence,
$h_\mu(T^\ell,\beta_k|\pi^{-1}(\mathcal{A}))=0$ for $k\in
\mathbb{N}$.

Now by Lemma \ref{RPF}, for $m,k\in \mathbb{N}$ we have
\begin{align*}
h_\mu(T^\ell,\alpha|\pi^{-1}(\mathcal{A}))&\le h_\mu(T^{\ell},\alpha\vee
\beta_k|\pi^{-1}(\mathcal{A}))\\
&=h_\mu(T^\ell,\beta_k|\pi^{-1}(\mathcal{A}))+h_\mu(T^\ell,\alpha|\bigvee_{t\in \mathbb{Z}}T^{t\ell}\beta_k\vee
\pi^{-1}(\mathcal{A}))\\
&=h_\mu(T^\ell,\alpha|\bigvee_{t\in \mathbb{Z}}T^{t\ell}\beta_k\vee \pi^{-1}(\mathcal{A}))\le
h_\mu(T^\ell,\alpha|\bigvee_{t\in \mathbb{Z}}T^{t\ell}\beta_k)\\
&=\inf_{n\ge 1} \frac{1}{n}H_\mu(\bigvee_{i=0}^{n-1}T^{-i\ell}\alpha|\bigvee_{t\in \mathbb{Z}}T^{t\ell}\beta_k)\le \frac{1}{m}H_\mu(\bigvee_{i=0}^{m-1}T^{-i\ell}\alpha|\bigvee_{t\in \mathbb{Z}}T^{t\ell}\beta_k)\\
&\le \frac{1}{m}H_\mu(\bigvee_{i=0}^{m-1}T^{-i\ell}\alpha|\beta_k).
\end{align*}
Fixing $m\in \mathbb{N}$ and letting $k\nearrow+\infty$ in the above
inequality we have
\[h_\mu(T^\ell,\alpha|\pi^{-1}(\mathcal{A}))\le
\frac{1}{m}H_\mu(\bigvee_{i=0}^{m-1}T^{-i\ell}\alpha|P_\mu(\pi))
\]
 since $H_\mu(\bigvee_{i=0}^{m-1}T^{-i\ell}\alpha|\beta_k)\searrow
H_\mu(\bigvee_{i=0}^{m-1}T^{-i\ell}\alpha|P_\mu(\pi))$ when $k\nearrow+\infty$. Then, letting $m\rightarrow +\infty$,
we ontain  $h_\mu(T^\ell,\alpha|\pi^{-1}(\mathcal{A}))
\le h_\mu(T^{\ell},\alpha|P_\mu(\pi))$. Thus, we have the equality \eqref{peq-en}.

\medskip
Next, let $a_n=\int_Y
H_{\mu_y}(\bigvee_{i=0}^{n-1}T^{-i\ell}\alpha)d \nu(y)$.  Since $T\mu_y=\mu_{Sy}$ for $\nu$-a.e. $y\in Y$ and $\nu$ is $S$-invariant, we have
\begin{align*}
a_n&=\int_Y
H_{\mu_{Sy}}(\bigvee_{i=0}^{n-1}T^{-i\ell}\alpha)d \nu(y)=\int_Y H_{T\mu_{y}}(\bigvee_{i=0}^{n-1}T^{-i\ell}\alpha)d \nu(y) \\
&=\int_Y H_{\mu_{y}}(T^{-1}(\bigvee_{i=0}^{n-1}T^{-i\ell}\alpha))d \nu(y)=\int_Y H_{\mu_{y}}(\bigvee_{i=1}^{n}T^{-i\ell}\alpha)d \nu(y).
\end{align*}
Moreover by the
monotone convergence Theorem, we have
\begin{align*}
\lim_{n\rightarrow +\infty} (a_{n+1}-a_n)&=\lim_{n\rightarrow
+\infty}\int_Y H_{\mu_y}(\alpha|\bigvee_{i=1}^n
T^{-i\ell}\alpha)d\nu(y)\\
&=\int_Y \lim_{n\rightarrow
+\infty}H_{\mu_y}(\alpha|\bigvee_{i=1}^n T^{-i\ell}\alpha)d \nu(y)\\
&=\int_Y h_\mu(T^\ell, \alpha,y)d \nu(y).
\end{align*}
Moreover, $$\lim_{n\rightarrow
+\infty}\frac{a_n}{n}=\lim_{n\rightarrow +\infty}
(a_{n+1}-a_n)=\int_Y h_\mu(T^\ell, \alpha,y)d \nu(y).$$ By
\eqref{meas3}, we have
$a_n=H_\mu(\bigvee_{i=0}^{n-1}T^{-i\ell}\alpha|P_\mu(\pi))$. Combing
this with \eqref{peq-en},
\begin{align*}
h_\mu(T^\ell,\alpha|\pi^{-1}(\mathcal{A}))&=h_\mu(T^\ell,\alpha|P_\mu(\pi))=\lim_{n\rightarrow
+\infty}\frac{a_n}{n}\\
&=\int_Y h_\mu(T^\ell, \alpha,y)d \nu(y).
\end{align*}
This completes the proof of the lemma.
\end{proof}

The following lemma is taken from \cite{R} (Lemma 3' in $\S$4 No.2).
\begin{lem}\label{rohlin} Suppose $\mu_y$ is non-atomic for $\nu$-a.e. $y\in Y$. If $B$ is a measurable set $X$ with $\mu_y(B)\ge r>0$ for
$\nu$-a.e. $y\in Y$, then for any $0\le s\le r$ there exists a measurable set $B_s$ such that $B_s\subseteq B$ and $\mu_y(B_s)=s$
for $\nu$-a.e. $y\in Y$.
\end{lem}

\begin{lem}\label{lem-5.3} Suppose $\mu_y$ is non-atomic for $\nu$-a.e. $y\in Y$. If $U_1,U_2\in \mathcal{B}$ with $\mu\times_Y \mu(U_1\times U_2)>0$, then there exist
a measurable set $A\subset Y$ with $\nu(A)>0$, a positive integer $r>2$ and a measurable partition $\alpha=\{ B_1,B_2,\cdots,B_r\}$ of $X$ such that
$\pi^{-1}_1(A)\cap B_i\subseteq U_i$, $i=1,2$ and $\mu_y(B_j)=\frac{1}{r}$, $j=1,2,\cdots,r$, for $\nu$-a.e. $y\in Y$.
\end{lem}
\begin{proof} Set $C_i=\{ y\in Y: \mu_y(U_i)>0\}$, $i=1,2$. Since
$$0<\mu\times_Y \mu(U_1\times U_2)=\int_Y \mu_y(U_1)\times \mu_y(U_2) d \nu(y),$$
we have $\nu(C_1\cap C_2)>0$. Thus there exist a positive integer $r>2$ and a measurable set $A\subseteq C_1\cap C_2$ such that
$\nu(A)>0$ and $\mu_y(U_i)\ge \frac{2}{r}$ for any $y\in A$, $i=1,2$.

Next we construct $B_j$, $j=1,2,\cdots,r$ by induction. Taking
$$D_1=\pi_1^{-1}(Y\setminus A)\cup (\pi^{-1}_1(A)\cap U_1),$$
then $\mu_y(D_1)\ge \frac{2}{r}$ for $\nu$-a.e. $y\in Y$. By Lemma \ref{rohlin}, there exists a measurable set $B_1\subseteq D_1$ such that $\mu_y(B_1)=\frac{1}{r}$
for $\nu$-a.e. $y\in Y$, and $B_1\cap \pi^{-1}_1(A) \subseteq D_1\cap \pi^{-1}_1(A)\subseteq U_1$.

Next, by the induction hypothesis, i.e., there are  measurable sets $B_k$, $(1\le k\le r)$
satisfying
\begin{itemize}
\item[{$(1)_k$}] $B_k\cap B_i=\emptyset$ for $1\le i\le k-1$;

\item[{$(2)_k$}] $\mu_y(B_k)=\frac{1}{r}$ for $\nu$-a.e. $y\in Y$;

\item[{$(3)_k$}] $B_k\cap \pi_1^{-1}(A)\subseteq U_k$ (when $k>2$, we set $U_k=X$).
\end{itemize}
If $k=r$, we are done. If $k<r$, we set
$$D_{k+1}=\begin{cases} \pi_1^{-1}(X\setminus A)\cup (\pi^{-1}_1(A)\cap (U_{2})\setminus  B_1) \ \ &\text{ if }k=1
\\
X\setminus \bigcup_{i=1}^k B_i \ \ &\text{ if }k>1.
\end{cases}
$$
It is clear that $\mu_y(D_{k+1})\ge \frac{1}{r}$ for $\nu$-a.e. $y\in Y$. Using Lemma \ref{rohlin} again, there exists a measurable set
$B_{k+1}\subseteq D_{k+1}$ such that $\mu_y(B_{k+1})=\frac{1}{r}$ for $\nu$-a.e. $y\in Y$. Then $B_{k+1}$ satisfies $(1)_{k+1}$, $(2)_{k+1}$ and $(3)_{k+1}$. Therefore, by induction, we obtain a measurable partition $\alpha=\{ B_1,B_2,\cdots,B_r\}$ of $X$ such that
    $\mu_y(B_k)=\frac{1}{r}$, $k=1,2,\cdots,r$, for $\nu$-a.e. $y\in Y$ and $B_i\cap \pi_1^{-1}(A)\subseteq U_i$, $i=1,2$.
This completes the proof of the lemma.\end{proof}
To prove our main theorem \ref{MTH1}, we also need the following consequence of Karpovsky-Milman-Alon's generalization of
the Sauer-Perles-Shelah lemma \cite{Al,KM,Sau,Sh}.
\begin{lem} (\cite{KM}) \label{com-lem} Given $r\ge 2$ and $\lambda > 1$ there is a constant $e > 0$ such that, for all $n \in \mathbb{N}$, if $S \subseteq \{1, 2, \cdots , r\}^{\{1,2,\cdots,n\}}$ satisfies
$|S|\ge ((r-1)\lambda)^n$ then there is an $I\subset \{1, 2,\cdots ,
n\}$ with $|I|\ge en$ and $S|_I = \{1, 2, \cdots, r\}^I$, i.e., for
any $u\in \{1,2,\cdots,r\}^I$ there is $s\in S$ with $s(j)=u(j)$ for
any $j\in I$.
\end{lem}

\subsection{Proof of Theorem \ref{MTH1}}

  Assume that
$ h_{top}(\phi,K)>0$. By Proposition \ref{VP} there exists $\mu\in \mathcal{E}_P^K(\Omega\times X)$  such
that $h_\mu(\phi)>0$. Since $K$ is a Borel subset of $\Omega\times X$, $K$ is also Polish space and the Borel
$\sigma$-algebra $\mathcal{B}_K$ of $K$ is just $\{ A\cap K: A\in  \mathcal{F}\times \mathcal{B}_X\}$. Thus
$(K,\mathcal{B}_K,\mu, \Phi)$ is an  ergodic Polish system and $\pi_\Omega: (K,\mathcal{B}_K,\mu,
\Phi)\rightarrow (\Omega, \mathcal{F},P,\theta)$ is a factor map between two Polish systems. Recall that the
entropy of RDS $\phi$ with respect to $\mu$  is given by
$$h_\mu(\phi):=h_\mu(\Phi|\pi_\Omega)=h_\mu(\Phi|\pi_\Omega^{-1}(\mathcal{F})).$$

Let $\mathcal{B}_P$ be the completion  of  $\mathcal{F}$ with  respect to  $P$. Let $\mathcal{B}_\mu$ be the
completion  of  $\mathcal{B}_K$ with  respect to  $\mu$. Then $\pi_\Omega:(K,\mathcal{B}_\mu,\mu,
\Phi)\rightarrow (\Omega,  \mathcal{B}_P,P,\theta)$  is a factor map between  two Lebesgue  systems. Since
$\mathcal{B}_P=\mathcal{F}$ (mod $P$) and  $\mathcal{B}_\mu=\mathcal{B}_K$ (mod  $\mu$), it  is  clear  that
$$
h_\mu(\Phi|\pi_\Omega^{-1}(\mathcal{B}_P))=h_\mu(\Phi|\pi_\Omega^{-1}(\mathcal{F}))=h_\mu(\phi).
$$ Let
$\Pi_K:(\bar{K},\bar{\mathcal{B}_{\mu}},\bar{\mu}, \bar{\Phi})\rightarrow (K,\mathcal{B}_\mu,\mu, \Phi)$ be
the  natural extension of $(K,\mathcal{B}_\mu,\mu, \Phi)$. Let $$\Pi_{\Omega}: (\bar{\Omega},
\bar{\mathcal{B}_{P}},\bar{P},\bar{\theta})\rightarrow (\Omega, \mathcal{B}_P,P,\theta)$$ be the natural
extension of $(\Omega, \mathcal{B}_P,P,\theta)$. Define $\bar{\pi}:\bar{K}\rightarrow \bar{\Omega}$ by
$\bar{\pi}(((\omega_i,x_i))_{i\in \mathbb{Z}})=(\omega_i)_{i\in \mathbb{Z}}$ for $((\omega_i,x_i))_{i\in
\mathbb{Z}}\in \bar{K}$.  Then $\bar{\pi}:(\bar{K},\bar{\mathcal{B}_{\mu}},\bar{\mu},\bar{\Phi})\rightarrow
(\bar{\Omega},\bar{\mathcal{B}_{P}},\bar{P},\bar{\theta})$ is a factor map between two invertible Lebesgue
systems and the following diagram is  commutative:
\begin{displaymath}
\begin{CD}
(K,\mathcal{B}_\mu,\mu,\Phi)  @< \Pi_K<<  (\bar{K},\bar{\mathcal{B}_{\mu}},\bar{\mu},\bar{\Phi})\\ @VV\pi_\Omega V
@VV\bar{\pi} V  \\ (\Omega,\mathcal{B}_P,P,\theta) @<  \Pi_\Omega<<
(\bar{\Omega},\bar{\mathcal{B}_{P}},\bar{P},\bar{\theta})\\
 \end{CD}
\end{displaymath}

Now, we  show
\begin{align}\label{na=en-1}
h_{\bar{\mu}}(\bar{\Phi}|\bar{\pi})=h_\mu(\phi)>0.
\end{align}

In  fact,  since $\pi_{\Omega}\circ \Pi_K=\Pi_\Omega \circ \bar{\pi}$, we have
\begin{align}\label{main-eq-1}
h_{\bar{\mu}}(\bar{\Phi}|\pi_{\Omega}\circ \Pi_K)=h_{\bar{\mu}}(\bar{\Phi}|\Pi_\Omega\circ  \bar{\pi}).
\end{align}
 By
Lemma \ref{GAR} and Lemma  \ref{na=zero}, we have
\begin{align}\label{main-eq-2}
h_{\bar{\mu}}(\bar{\Phi}|\pi_{\Omega}\circ
\Pi_K)=h_{\bar{\mu}}(\bar{\Phi}|\Pi_K)+h_\mu(\Phi|\pi_{\Omega})=h_\mu(\Phi|\pi_{\Omega})=h_\mu(\phi)
 \end{align}
 and
\begin{align}\label{main-eq-3} h_{\bar{\mu}}(\bar{\Phi}|\Pi_\Omega\circ
\bar{\pi})=h_{\bar{\mu}}(\bar{\Phi}|\bar{\pi})+h_{\bar{P}}(\bar{\theta}|\Pi_\Omega)=h_{\bar{\mu}}(\bar{\Phi}|\bar{\pi}).
\end{align} Using \eqref{main-eq-1},\eqref{main-eq-2} and  \eqref{main-eq-3}, we get \eqref{na=en-1}.

Next, let $P_{\bar{\mu}}(\bar{\pi})$ be the relative Pinsker $\sigma$-algebra of
$(\bar{K},\bar{\mathcal{B}_{\mu}},\bar{\mu},\bar{\Phi})$ relative  to $\bar{\pi}$. Since
${\bar{\Phi}}^{-1}(P_{\bar{\mu}}(\bar{\pi}))=P_{\bar{\mu}}(\bar{\pi})$  and
$$\bar{\pi}^{-1}(\mathcal{\bar B}_{{P}})\subseteq P_{\bar{\mu}}(\bar{\pi})\subseteq \bar{\mathcal{B}_{\mu}},$$
there exist an invertible  Lebesgue system  $(Y,\mathcal{D},\nu,S)$ and  two factor maps
$$\pi_1:(\bar{K},\bar{\mathcal{B}_{\mu}},\bar{\mu},\bar{\Phi})\rightarrow (Y,\mathcal{D},\nu,S),\
\pi_2:(Y,\mathcal{D},\nu,S)\rightarrow  (\bar{\Omega},\bar{\mathcal{B}_{P}},\bar{P},\bar{\theta})$$ between
invertible  Lebesgue systems such that  $\pi_2\circ \pi_1=\bar{\pi}$  and
$\pi_1^{-1}(\mathcal{D})=P_{\bar{\mu}}(\bar{\pi})$ (mod $\bar{\mu}$). That is, $\pi_1$ is the Pinsker factor
of $(\bar{K},\bar{\mathcal{B}_{\mu}},\bar{\mu},\bar{\Phi})$ relative  to $\bar{\pi}$.

Let  $\bar{\mu}=\int_Y \bar{\mu}_y d \nu(y)$ be the disintegration
of $\bar{\mu}$ relative to the  factor $(Y,\mathcal{D},\nu,S)$. Let
$$(\bar{K}\times \bar{K},\bar{\mathcal{B}_{\mu}}\times
\bar{\mathcal{B}_{\mu}},\bar{\mu} \times_Y\bar{\mu},
\bar{\Phi}\times \bar{\Phi})$$  be the product of
$(\bar{K},\bar{\mathcal{B}_{\mu}},\bar{\mu},\bar{\Phi})$ with itself
relative to a  factor $(Y,\mathcal{D},\nu,S)$. Recall that
$$(\bar{\mu}\times_Y \bar{\mu})(B)=\int_{Y}(\bar{\mu}_y\times
\bar{\mu}_y)(B) \, d \nu(y), \ B\in \bar{\mathcal{B}_{\mu}}\times
\bar{\mathcal{B}_{\mu}}.$$ Since $(K,\mathcal{B}_{\mu},\mu,\Phi)$ is
ergodic, we have that
$(\bar{K},\bar{\mathcal{B}_{\mu}},\bar{\mu},\bar{\Phi})$ is ergodic.
By \eqref{na=en-1}, $h_{\bar{\mu}}(\bar{\Phi}|\bar{\pi})>0$.  By
Lemma \ref{key-lem}, we have \begin{enumerate} \item[{(a1).}]
$\bar{\mu}_y$ is non-atomic for $\nu$-a.e. $y\in Y$.

\item[{(a2).}] $(\bar{K}\times \bar{K},\bar{\mathcal{B}_{\mu}}\times \bar{\mathcal{B}_{\mu}},\bar{\mu}
    \times_Y\bar{\mu}, \bar{\Phi}\times \bar{\Phi})$ is ergodic.
\end{enumerate}

 Since $(X,d)$ is a separable metric space, there exists a countable dense set
$X'$ of $X$. Let $\Gamma=\{ \overline{B}(x,t):x\in X' \text{ and }
t\in \mathbb{Q}, t>0\}$, where $\overline{B}(x,t)=\{z\in X:
d(z,x)\le t\}$. Clearly, $\Gamma$ is a countable set and each
element of $\Gamma$ is a non-empty closed set of $X$. Let $\Theta=\{
\{U_1,U_2\}: U_1,U_2\in \Gamma, d(U_1,U_2)>0\}$. Then $\Theta$ is
also a countable set. For $U\subseteq X$, we define
\begin{align}\label{blde}
\widetilde{U}=\{ ((\omega_i,x_i))_{i\in \mathbb{Z}}\in \overline{K}: x_0\in U\}.
\end{align}

Recall that $\Delta_{\bar{K}}=\{ (\bar{k},\bar{k}):\bar{k}\in
\bar{K}\}$. It is clear that when $\bar{\mu}_y$ is non-atomic, i.e.,
$\bar{\mu}_y(\{  \bar{k}\})=0$ for $\bar{k}\in \bar{K}$, we have
$$\bar{\mu}_y\times
\bar{\mu}_y(\Delta_{\bar{K}})=\int_{\bar{K}}\left(
\int_{\bar{K}}1_{\Delta_{\bar{K}}}(\bar{k}_1,\bar{k}_2)
d\bar{\mu}_y(\bar{k}_2) \right)
d\bar{\mu}_y(\bar{k}_1)=\int_{\bar{K}}\bar{\mu}_y(\{\bar{k}_1\}) d
\bar{\mu}_y(\bar{k}_1)=0
$$ where
$1_{\Delta_{\bar{K}}}$ is the characterization function  of
$\Delta_{\bar{K}}$. By (a1), for $\nu$-a.e. $y\in Y$, $\bar{\mu}_y$
is non-atomic. Thus for  $\nu$-a.e. $y\in Y$, $\bar{\mu}_y\times
\bar{\mu}_y(\Delta_{\bar{K}})=0$. Moreover,
\begin{align*}
\bar{\mu}\times_Y \bar{\mu}(\Delta_{\bar{K}})=\int_Y
\bar{\mu}_y\times \bar{\mu}_y(\Delta_{\bar{K}}) d\nu(y)=0.
\end{align*}
Thus $\bar{\mu}\times_Y \bar{\mu}(\bar{K}\times
\bar{K}\setminus\Delta_{\bar{K}})=1$. Moreover, since $\bar{K}\times
\bar{K}\setminus\Delta_{\bar{K}}=\bigcup_{i\in
\mathbb{Z}}\bigcup_{\{U_1,U_2\}\in \Theta} (\bar{\Phi}\times
\bar{\Phi})^i \left( \widetilde{U_1}\times \widetilde{U_2} \right)$
and $\bar{\mu}\times_Y \bar{\mu}$ is $\bar{\Phi}\times
\bar{\Phi}$-invariant, there exists $\{U_1,U_2\}\in \Theta$ such
that $\bar{\mu}\times_Y \bar{\mu}(\widetilde{U_1}\times
\widetilde{U_2})>0$.

\medskip
In the following we want to show that $\{U_1,U_2\}$ is a weak
Horseshoe of $(\phi,K)$. First, it is clear that $U_1$ and $U_2$
are non-empty, closed and bounded subsets of $X$ and $d(U_1,U_2)>0$ by the definition
of $\Theta$. We divide the remainder of the proof into three steps.

\medskip \noindent{\it Step 1.}
Since $\bar{\mu}_y$ is non-atomic for $\nu$-a.e. $y\in Y$ and
$\bar{\mu}\times_Y \bar{\mu}(\widetilde{U_1}\times
\widetilde{U_2})>0$, by Lemma \ref{lem-5.3}, there exist a measurable
set $A\subset Y$ with $\nu(A)>0$, a positive integer $r>2$, and a
measurable partition $\alpha=\{ B_1,B_2,\cdots,B_r\}$ of $\bar{K}$
such that $\pi^{-1}_1(A)\cap B_i\subseteq \widetilde{U_i}$, $i=1,2$
and $\bar{\mu}_y(B_j)=\frac{1}{r}$, $j=1,2,\cdots,r$, for $\nu$-a.e.
$y\in Y$.

By Lemma \ref{key-lem-1}
$$\lim_{m\rightarrow +\infty} h_{\bar{\mu}}(\bar{\Phi}^m,\alpha|\bar{\pi}^{-1}(\bar{\mathcal{B}_P}))=H_{\bar{\mu}}(\alpha|P_{\bar{\mu}}(\bar{\pi}))
=\sum_{j=1}^r \int_Y -\bar{\mu}_y(B_j)\log \bar{\mu}_y(B_j) d
\nu(y)=\log r.$$ Thus, there is an $\ell>0$ such that
$h_{\bar{\mu}}(\bar{\Phi}^\ell,\alpha|\bar{\pi}^{-1}(\bar{\mathcal{B}_P}))>\nu(Y\setminus
A) \cdot \log r+\nu(A)\cdot \log (r-1)$.

Let
$c:=\frac{1}{3}\left(h_{\bar{\mu}}(\bar{\Phi}^\ell,\alpha|\bar{\pi}^{-1}(\bar{\mathcal{B}_P}))-\left(
\nu(Y\setminus A) \cdot \log r+\nu(A)\cdot \log
(r-1)\right)\right)$. Then $c>0$. Recall that we defined $\nu$-a.e.
on $Y$ a function
$$h_{\bar{\mu}}(\bar{\Phi}^\ell,\alpha,y):=\lim_{n\rightarrow
+\infty}H_{\bar{\mu}_y}(\alpha|\bigvee_{i=1}^n
\bar{\Phi}^{-i\ell}\alpha).$$ By Lemma \ref{lem-5.1}, we have
$\int_Y
h_{\bar{\mu}}(\bar{\Phi}^\ell,\alpha,y)d\nu(y)=h_{\bar{\mu}}(\bar{\Phi}^\ell,\alpha|\bar{\pi}^{-1}(\bar{\mathcal{B}_P}))$.
For $y\in Y$, we set
$$\delta(y):=(h_{\bar{\mu}}(\bar{\Phi}^\ell,\alpha,y)-\log (r-1))\cdot 1_A(y).$$
Since $0\le h_{\bar{\mu}}(\bar{\Phi}^\ell,\alpha,y)\le \log r$,
$\delta(y)$ is a bounded measurable function on $Y$ and
\begin{align}\label{key-yle}
1_A(y)\ge \frac{\delta (y)}{\log (\frac{r}{r-1})}.
\end{align}
Now,
\begin{align*}
\int_Y \delta(y) d\nu(y)&=\int_A
h_{\bar{\mu}}(\bar{\Phi}^\ell,\alpha,y)d\nu(y)-\nu(A)\log (r-1)\\
&=\int_Y h_{\bar{\mu}}(\bar{\Phi}^\ell,\alpha,y)d\nu(y)-
\int_{Y\setminus A}
h_{\bar{\mu}}(\bar{\Phi}^\ell,\alpha,y)d\nu(y)-\nu(A)\log (r-1)\\
&\ge
h_{\bar{\mu}}(\bar{\Phi}^\ell,\alpha|\bar{\pi}^{-1}(\bar{\mathcal{B}_P}))-
\nu(Y\setminus A) \log r-\nu(A)\log (r-1)\\
&=3c.
\end{align*}
By the Birkhoff ergodic Theorem,
$\frac{1}{m}\sum_{i=0}^{m-1}1_A(S^{i\ell}y)$ converges $\nu$-a.e. to
a function $1_A^*\in L^1(\nu)$;
$\frac{1}{m}\sum_{i=0}^{m-1}\delta(S^{i\ell}y)$ converges $\nu$-a.e.
to a function $\delta^*\in L^1(\nu)$ and $\int_Y \delta^*(y) d
\nu(y)=\int_Y \delta(y) d \nu(y)\ge 3c$. Note that  for any
$i\in \mathbb{Z}$, $\bar{\Phi}^i \bar{\mu}_y=\bar{\mu}_{S^iy}$ and
$\bar{\Phi}^i \alpha \subseteq \bar{\mathcal{B}}_{\bar{\mu}_y}$ for
$\nu$-a.e. $y\in Y$, where $\bar{\mathcal{B}}_{\bar{\mu}_y}$ is the
completion of $\bar{\mathcal{B}_{\mu}}$ under $\bar{\mu}_y$.

Next we define the measurable subset $D$ of $Y$ such that $y\in D$
if and only if the following holds
\begin{itemize}
\item  $\lim_{m\rightarrow +\infty}\frac{1}{m}\sum_{i=0}^{m-1}1_A(S^{i\ell}y)=1_A^*(y)$;

\item  $\lim_{m\rightarrow +\infty}\frac{1}{m}\sum_{i=0}^{m-1}\delta(S^{i\ell}y)=\delta^*(y)\ge 2c$;

\item  for any
$i\in \mathbb{Z}$, $\bar{\Phi}^i \bar{\mu}_y=\bar{\mu}_{S^iy}$ and
$\bar{\Phi}^i \alpha \subseteq \bar{\mathcal{B}}_{\bar{\mu}_y}$.
\end{itemize}
Since $\int_Y \delta^*(y)d\nu(y)\ge 3c$, one has $\nu(D)>0$. For $y\in D$,
let
$$\mathcal{S}(y)=\{ \ell k: k\in \mathbb{Z}_+, S^{\ell k}y\in
A\}:=\{ a_1(y)<a_2(y)<\cdots\}.$$ Then $S^{a_i(y)}y\in A$ for $i=1,2,\cdots$. Put $c_1=\frac{c}{\ell \log
(\frac{r}{r-1})}$. Using \eqref{key-yle},
\begin{align}\label{density}
\begin{aligned}
&\hskip0.5cm \lim_{m\rightarrow +\infty} \frac{|\mathcal{S}(y)\cap \{0,1,\cdots,m-1\}|}{m}\\
&=\lim_{m\rightarrow +\infty} \frac{1}{\ell m} \sum_{i=0}^{m-1}1_A(S^{i\ell}y)\ge \lim_{m\rightarrow +\infty} \frac{1}{\ell m} \sum_{i=0}^{m-1}\frac{\delta(S^{i\ell}y)}{\log (\frac{r}{r-1})}\\
&=\frac{\delta^*(y)}{\ell \log (\frac{r}{r-1})}\ge \frac{3c}{\ell \log (\frac{r}{r-1})}\\
&=3c_1>0.
\end{aligned}
\end{align}
Let $E_0=D$ and $E_i=S^{-i}D\setminus \bigcup_{j=0}^{i-1}S^{-j}D$
for $i=1,2,\cdots$. Since $(Y,\mathcal{D},\nu,S)$ is ergodic and
$\nu(D)>0$,
$E:=\bigcup_{i=0}^{+\infty}S^{-i}D=\bigcup_{i=0}^{+\infty}E_i$ has
full measure, i.e., $\nu(E)=1$.

\medskip \noindent{\it Step 2.} Given $z\in E$. There exists unique $i\in \{0,1,2,\cdots\}$ and unique $y\in D$ such that
$z\in E_i$ and $z=S^{-i}(y)$. Put $a_j(z)=i+a_j(y)$ for $j=1,2,\cdots$ and
$\mathcal{S}(z)=\{a_1(z)<a_2(z)<\cdots\}$.
Clearly,
\begin{align}\label{s-density}
\lim_{m\rightarrow +\infty} \frac{|\mathcal{S}(z)\cap
\{0,1,\cdots,m-1\}|}{m}=\lim_{m\rightarrow +\infty}
\frac{|\mathcal{S}(y)\cap \{0,1,\cdots,m-1\}|}{m}\ge 3c_1>0
\end{align}
by \eqref{density}.

For $k\in \mathbb{N}$, let
$$M_z(k):=\{ (j_1,j_2,\cdots,j_k)\in \{1,2,\cdots,r\}^{\{1,2,\cdots,k\}}:\bar{\mu}_z(\pi_1^{-1}(z)\cap (\bigcap_{t=1}^k \bar{\Phi}^{-a_t(z)}B_{j_t}))>0\}.$$
Since $\bar{\mu}_z(\pi_1^{-1}(z))=1$, $\bar{\Phi}^u \bar{\mu}_y=\bar{\mu}_{S^uy}$ for any $u\in
\mathbb{Z}$ and $\ell |a_t(y)$ for any $t\in \mathbb{N}$, we have
\begin{align*}
\log |M_z(k)|&\ge H_{\bar{\mu}_z}(\bigvee_{t=1}^k
\bar{\Phi}^{-a_t(z)}\alpha)=H_{\bar{\mu}_{S^{-i}y}}(\bigvee \limits_{t=1}^k \bar{\Phi}^{-(i+a_t(y))}\alpha) \\
&=H_{\bar{\mu}_{y}}(\bigvee_{t=1}^k \bar{\Phi}^{-a_t(y)}\alpha
)=H_{\bar{\mu}_y}(\bar{\Phi}^{-a_1(y)}\alpha|\bigvee_{t=2}^k
\bar{\Phi}^{-a_t(y)}\alpha)+H_{\bar{\mu}_y}(\bigvee_{t=2}^k
\bar{\Phi}^{-a_t(y)}\alpha)\\&=H_{\bar{\mu}_{S^{a_1(y)}y}}(\alpha|\bigvee_{t=2}^k
\bar{\Phi}^{-(a_t(y)-a_1(y))}\alpha)+H_{\bar{\mu}_y}(\bigvee_{t=2}^k
\bar{\Phi}^{-a_t(y)}\alpha)\\
&\ge
h_{\bar{\mu}}(\bar{\Phi}^\ell,\alpha,S^{a_1(y)}y)+H_{\bar{\mu}_y}(\bigvee_{t=2}^k
\bar{\Phi}^{-a_t(y)}\alpha)\\
&\cdots \cdots \cdots \cdots \cdots \cdots\\
&\ge \sum_{t=1}^k h_{\bar{\mu}}(\bar{\Phi}^\ell,\alpha,S^{a_t(y)}y).
\end{align*}
Since $\frac{a_k(y)}{\ell}\ge k-1$ and
$$\lim_{k\rightarrow +\infty}
\frac{1}{\frac{a_k(y)}{\ell}+1}\sum_{j=0}^{\frac{a_k}{\ell}}\delta(S^{\ell
j }y)=\delta^*(y)\ge 2c,$$ there exists $N_y\in \mathbb{N}$ such
that when $k\ge N_y$,
$\frac{1}{\frac{a_k(y)}{\ell}+1}\sum \limits_{j=0}^{\frac{a_k}{\ell}}\delta(S^{\ell
j }y)\ge c$. Thus for $k\ge N_y$,
\begin{align*}
\log |M_z(k)|&\ge \sum_{t=1}^k
h_{\bar{\mu}}(\bar{\Phi}^\ell,\alpha,S^{a_t(y)}y)\\
&=k\cdot \log (r-1)+\sum_{j=0}^{\frac{a_k}{\ell}}(h_{\bar{\mu}}(\bar{\Phi}^\ell,\alpha,S^{\ell j}y)-\log(r-1))\cdot 1_A(S^{\ell j}y)\\
&=k\cdot \log (r-1)+\sum_{j=0}^{\frac{a_k}{\ell}}\delta(S^{\ell j
}y)\\
&\ge k\cdot \log (r-1)+(\frac{a_k(y)}{\ell}+1)c\\
&\ge k\cdot \log (r-1)+k c.
\end{align*}
Let $\lambda=2^c$. Then $\lambda>1$ and $$|M_z(k)|\ge
((r-1)\lambda)^k$$ for $k\ge N_y$. By Lemma \ref{com-lem}, there
exists constant $e>0$ ($e$ is just dependent on $r$ and $\lambda$)
such that for $k\ge N_y$ there exists $I_{k,z}\subseteq
\{1,2,\cdots,k \}$ such that $|I_{k,z}|\ge ek$ and
$M_z(k)|_{I_{k,z}}=\{1,2,\cdots,r\}^{I_{k,z}}$.

Put $L_{k,z}=\{ a_t(z):t\in I_{k,z}\}$. Then for $k\ge
\max\{N_y,\frac{1}{e}\}$, we have $|L_{k,z}|=|I_{k,z}|\ge ek\ge 1$,
$J_{k,z}\subseteq \{ 0,1,\cdots,a_k(z)\}$ and for any $s\in
\{1,2,\cdots,r\}^{L_{k,z}}$, one has
\begin{align*}
\bar{\mu}_z(\pi_1^{-1}(z)\cap (\bigcap_{j\in L_{k,z}}
\bar{\Phi}^{-j}B_{s(j)}))>0.
\end{align*}
By \eqref{s-density}, there exists $W_z\in \mathbb{N}$ such that
when $m\ge W_z$, one has
$$|\mathcal{S}(z)\cap \{0,1,\cdots,m-1\}|\ge c_1m \ge \max\{N_y,\frac{1}{e}\}$$

For $m\ge W_z$, let $k_m(z)=|\mathcal{S}(z)\cap
\{0,1,\cdots,m-1\}|$. Then $k_m(z)\ge c_1m \ge
\max\{N_y,\frac{1}{e}\}$ and
$$\mathcal{S}(z)\cap \{0,1,\cdots,m-1\}=\{a_1(z),a_2(z),\cdots,a_{k_m(z)}(z)\}.$$
Let $J_{m,z}=L_{k_m(z),z}$. Then $J_{m,z}\subseteq
\{a_1(z),a_2(z),\cdots,a_{k_m(z)}(z)\}\subseteq \{0,1,\cdots,m-1\}$,
$$|J_{m,z}|=|L_{k_m(z),z}|\ge ek_m(z)\ge ec_1m$$ and for  any $s\in
\{1,2,\cdots,r\}^{J_{m,z}}$, one has
\begin{align*}
\bar{\mu}_z(\pi_1^{-1}(z)\cap (\bigcap_{j\in J_{m,z}}
\bar{\Phi}^{-j}B_{s(j)}))>0.
\end{align*}

\medskip \noindent{\it Step 3.} Let $b=ec_1$. Then $b>0$ and for $P$-a.e. $\omega\in \Omega$,
there exists $M_{b,\omega}\in \mathbb{N}$ such that for any natural
number $m\ge M_{b,\omega}$ we can find $J_m\subset
\{0,1,\cdots,m-1\}$ with $|J_m|\ge bm$, and for any $s\in
\{1,2\}^{J_m}$, there exists $x_s\in K(\omega)$ with
$\phi(j,\omega,x_s)\in U_{s(j)}\cap K(\theta^j \omega)$ for any
$j\in J_m$.

\medskip Note that $$\Pi_\Omega\circ \pi_2:(Y,\mathcal{D},\nu,S)\rightarrow
(\Omega,\mathcal{B}_P, P, \theta)$$ is a factor map between
two Lebesgue systems. Since $\nu(E)=1$, there exists $\Omega_1\in
\mathcal{B}_P$ satisfying $P(\Omega_1)=1$ and $(\Pi_\Omega\circ
\pi_2)^{-1}(\omega)\cap E\neq \emptyset$ for  each $\omega\in
\Omega_1$.

Given $\omega\in \Omega_1$. By the definition of $\Omega_1$, there
exists $z\in E$ such that $\Pi_\Omega\circ \pi_2(z)=\omega$.
Let $\pi_2(z)=(\omega_i(z))_{i\in \mathbb{Z}}\in \bar{\Omega}$. Then $\omega=\Pi_\Omega ((\omega_i(z))_{i\in \mathbb{Z}})=\omega_0(z)$.
By Step
2, we can find $W_z\in \mathbb{N}$ such that for any $m\ge W_z$,
there exists
$$J_{m,z}\subseteq \{a_1(z),a_2(z),\cdots,a_{k_m(z)}(z)\}\subseteq
\{0,1,\cdots,m-1\}$$ with $|J_{m,z}|\ge bm$, and for  any $s\in
\{1,2,\cdots,r\}^{J_{m,z}}$, one has
\begin{align}\label{key-non-empty-2}
\bar{\mu}_z(\pi_1^{-1}(z)\cap (\bigcap_{j\in J_{m,z}}
\bar{\Phi}^{-j}B_{s(j)}))>0.
\end{align}

We take $M_{b,\omega}=W_z$. For $m\ge M_{b,\omega}$, take $J_m=J_{m,z}$. Then $J_m\subseteq \{0,1,\cdots,m-1\}$ with $|J_m|\ge bm$.
Next we are ready to show that for any $s\in
\{1,2\}^{J_m}$, there exists $x_s\in K(\omega)$ with
$\phi(j,\omega,x_s)\in U_{s(j)}\cap K(\theta^j \omega)$ for any
$j\in J_m$.

\medskip
Given $s\in
\{1,2\}^{J_{m}}$. By \eqref{key-non-empty-2} we know $\bar{\mu}_z(\pi_1^{-1}(z)\cap (\bigcap_{j\in J_{m}}
\bar{\Phi}^{-j}B_{s(j)}))>0$. Thus $\pi_1^{-1}(z)\cap (\bigcap_{j\in J_{m}}
\bar{\Phi}^{-j}B_{s(j)})\neq \emptyset$ and we can find $((\omega_i^s,x_i^s))_{i\in \mathbb{Z}}\in \pi_1^{-1}(z)\cap (\bigcap_{j\in J_{m}}
\bar{\Phi}^{-j}B_{s(j)})\subseteq \bar{K}$. On the hand  $$\bar{\pi}(((\omega_i^s,x_i^s))_{i\in \mathbb{Z}})=(\omega_i^s)_{i\in \mathbb{Z}}\in \bar{\Omega}.$$
On the other hand
$$\bar{\pi}(((\omega_i^s,x_i^s))_{i\in \mathbb{Z}})=\pi_2\circ \pi_1(((\omega_i^s,x_i^s))_{i\in \mathbb{Z}})
=\pi_2(z)=(\omega_i(z))_{i\in \mathbb{Z}}\in \bar{\Omega}.$$
Hence $(\omega_i^s)_{i\in \mathbb{Z}}=(\omega_i(z))_{i\in \mathbb{Z}}$. Particularly, $\omega_0^s=\omega_0(z)=\omega$. Thus $x_0^s\in K(\omega)$ since
$(\omega,x_0^s)=(\omega_0,x_0^s)\in K$. Take $x_s=x_0^s$. Then $x_s\in K(\omega)$.
Recall that $$\bar{K}=\{ ((\omega_i,x_i))_{i\in \mathbb{Z}}\in K^{\mathbb{Z}}: \Phi(\omega_i,x_i)=(\omega_{i+1},x_{i+1}) \text{ for }i\in \mathbb{Z}\}.$$
Since $((\omega_i^s,x_i^s))_{i\in \mathbb{Z}}\in \bar{K}$ and $(\omega_0^s,x_0^s)=(\omega,x_s)$,
one has $\omega_i^s=\theta^i\omega$ and $x_i^s=\phi(i,\omega,x_s)\in K(\theta^i\omega)$ for any $i\in \mathbb{Z}$.
Finally, we show that for any $j\in J_m$,
$\phi(j,\omega,x_s)\in U_{s(j)}\cap K(\theta^j \omega)$.

\medskip

\medskip
Given $j\in J_m$. Since $z\in E$, there exists unique $i_*\ge 0$  and unique $y\in D$ such that $z\in E_{i_*}$ and $z=S^{-i_*}(y)$.
Note that $$J_m=J_{m,z}\subseteq \{a_1(z),a_2(z),\cdots,a_{k_m(z)}(z)\}=\{ i_*+a_1(y),i_*+a_2(y),\cdots,i_*+a_{k_m(z)}(y)\},$$
we have $j-i_*\in \{a_1(y),a_2(y),\cdots,a_{k_m(z)}(y)\}$ and so $S^{j-i_*}(y)\in A$, that is, $S^j(z)\in A$.
Since $((\omega_i^s,x_i^s))_{i\in \mathbb{Z}}\in \pi_1^{-1}(z)\cap
\bar{\Phi}^{-j}B_{s(j)})$, one has
\begin{align*}
\bar{\Phi}^j(((\omega_i^s,x_i^s))_{i\in \mathbb{Z}})&\in \bar{\Phi}^j(\pi_1^{-1}(z))\cap B_{s(j)}=\pi_1^{-1}(S^jz)\cap B_{s(j)}\\
&\subseteq \pi_1^{-1}(A)\cap B_{s(j)}\subseteq \widetilde{U_{s(j)}}.
\end{align*}
Moreover as $\bar{\Phi}^j(((\omega_i^s,x_i^s))_{i\in \mathbb{Z}})=((\omega_{i+j}^s,x_{i+j}^s))_{i\in \mathbb{Z}}$, using \eqref{blde}
we know $x_j^s\in U_{s(j)}$. Combing this with $x_j^s=\phi(j,\omega,x_s)\in K(\theta^j\omega)$, this shows
$\phi(j,\omega,x_s)\in U_{s(j)}$ and completes the proof of the theorem.
\hfill$\square$

\section{Proof of Theorem \ref{MTH3}}
In this section, we prove a stronger result than Theorem \ref{MTH1} when  $(\Omega, \mathcal{F}, {P},\theta)$ is a
{compact metric system} and $K(\omega),\,\omega\in\Omega$ is { a strongly compact random set}. We show that there is a full horseshoe for $(\phi, K)$ instead of a weak horseshoe. As a consequence, we have a full horseshoe for a continuous deterministic dynamical systems $\phi$ on a compact invariant set $K$ with positive topological entropy.

 \vskip0.1in

\noindent{\bf Proof of Theorem \ref{MTH3}. } By Theorem~\ref{MTH1} there exist  subsets $U_1,U_2$ of $X$, a constant $b>0$ and $\Omega_0\in \mathcal{F}$ with $P(\Omega_0)=1$ such that
\begin{enumerate}
\item $U_1$ and $U_2$ are non-empty closed bounded subsets of $X$ and
$d(U_1,U_2)>0$.

\item  for each $\omega\in \Omega_0$,
there exists $M_{b,\omega}\in \mathbb{N}$ such that for any natural
number $m\ge M_{b,\omega}$, there is a subset $J_m(w_*)\subset
\{0,1,2,\cdots,m-1\}$ with $|J_m(\omega)|\ge bm$(positive density), and for any $s\in
\{1,2\}^{J_m(\omega}$, there exists $x_s\in K(\omega)$ with
$\phi(j,\omega,x_s)\in U_{s(j)}\cap K(\theta^j \omega)$ for any
$j\in J_m(\omega)$.
\end{enumerate}

Let $\mathbb{N}_0=\mathbb{N}\cup \{0\}$ and $T:\Omega\times \{0,1\}^{\mathbb{N}_0}\rightarrow \Omega\times \{0,1\}^{\mathbb{N}_0}$ be the map such that  $T(\omega,u)=(\theta\omega, \sigma u)$
for $\omega\in\Omega$ and $u=(u(n))_{n\in \mathbb{N}_0}\in \{0,1\}^{\mathbb{N}}$, where $\sigma u=(u_{n+1})_{n\in \mathbb{N}_0}$ is the left shift map on
$\{0,1\}^{\mathbb{N}_0}$. Clearly $\Omega\times \{0,1\}^{\mathbb{N}_0}$ is a compact metric space and $T$ is a continuous self-map on $\Omega\times \{0,1\}^{\mathbb{N}_0}$.

Consider the subset $Y$ of $\Omega\times \{0,1\}^{\mathbb{N}_0}$ such that $(\omega,u)\in Y$ if and only if
 for any $s\in \{1,2\}^J$,put $J=\{n\in \mathbb{N}_0:u(n)=1\}$,  there exists $x_s\in K(\omega)$ with
$\phi(j,\omega,x_s)\in U_{s(j)}\cap K(\theta^j \omega)$ for any
$j\in J$. We have the following claim.

\medskip

\noindent{\it Claim.} $T(Y)\subseteq Y$ and $Y$ is a closed subset of $\Omega\times \{0,1\}^{\mathbb{N}_0}$.

\begin{proof}[Proof of claim] If $(\omega,u)\in Y$, let $J=\{n\in \mathbb{N}_0:u(n)=1\}$ and $J_\sigma=\{n\in \mathbb{N}_0:(\sigma u)(n)=1\}$. Then
$J_\sigma+1\subset J$. Now for any $t\in \{1,2\}^{J_\sigma}$, we can find $s\in \{1,2\}^J$ such that $t(n)=s(n+1)$ for $n\in J_\sigma$. As $(\omega,u)\in Y$,
there exists $x_s\in K(\omega)$ with $\phi(j,\omega,x_s)\in U_{s(j)}\cap K(\theta^j \omega)$ for any $j\in J$.
Thus let $x_t=\phi(\omega)(x_s)$. Then $x_t\in K(\theta \omega)$ and $\phi(j,\theta\omega,x_t)=\phi(j+1,\omega,x_s)\in U_{s(j+1)}\cap K(\theta^j \theta\omega)$ for any $j\in J_\sigma$.
Hence $(\theta\omega,\sigma u)\in Y$. This implies $T(Y)\subseteq Y$.

Next let $(\omega^i,u^i)\in Y$ and $\lim_{i\rightarrow +\infty}(\omega^i,u^i)=(\omega,u)$ in $\Omega\times \{0,1\}^{\mathbb{N}_0}$. Let $J_i=\{n\in \mathbb{N}_0:u^i(n)=1\}$ for $i\in \mathbb{N}$
and $J=\{n\in \mathbb{N}_0: u(n)=1\}$. If $J=\emptyset$, then it is clear that $(\omega,u)\in Y$ by the definition of $Y$. Now suppose $J\neq\emptyset$, and let $N=\min\{n\in J\}$.

Given $s\in \{1,2\}^J$. Note that $\lim_{i\rightarrow +\infty}u^i=u$. There exist $1\le i_1<i_2<\cdots$ such that
$$J_{i_k}\cap \{0,1,\cdots,N+r\}=J\cap \{0,1,\cdots,N+r\}$$ for each $1\le r\le k$.
Now for each $k$, we take $s^k\in \{1,2\}^{J_{i_k}}$ such that $s^k(j)=s(j)$ for $j\in J\cap \{0,1,\cdots,N+k\}$.
As $(\omega^{i_k},u^{i_k})\in Y$, there exists $x_{s^k}\in K(\omega^{i_k})$ with
$\phi(j,\omega^{i_k},x_{s^k})\in U_{s^k(j)}\cap K(\theta^j \omega^{i_k})$ for any
$j\in J_{i_k}$.  Without of loss generality, we assume that the limit $\lim_{k\rightarrow +\infty}x_{s_k}$ exists (if necessarily we take subsequence) and
let $x_s:=\lim_{k\rightarrow +\infty}x_{s^k}$.

Firstly, as  $x_{s^k}\in K(\omega^{i_k})$ and the function
$\omega\mapsto\inf_{y\in M(\omega)}d(x,y)$
is lower semi-continuous  for any $x\in X$, one has
\begin{align*}
\inf_{y\in K(\omega)} d(x_s,y)&\le \liminf_{k\rightarrow +\infty} \inf_{y\in K(\omega^{i_k})} d(x_s,y)\\
&\le  \liminf_{k\rightarrow +\infty} \inf_{y\in K(\omega^{i_k})} \big( d(x_s,x_{s^k})+d(x_{s^k},y)\big)\\
&=0.
\end{align*}
Thus $x_s\in K(\omega)$.

Given $n\in J$. As the map $(\omega,x)\mapsto \phi(\omega)x$ is a continuous map, one has $$\phi(n,\omega,x_s)=\lim_{k\rightarrow +\infty}\phi(n,\omega^{i_k},x_{s^k}).$$ Combing this with the fact that
$$\phi(n,\omega^{i_k},x_{s^k})\in U_{s^k(n)}\cap K(\theta^n \omega^{i_k})=U_{s(n)}\cap K(\theta^n \omega^{i_k})$$ for $k\ge n$ and the function
$\omega\mapsto\inf_{y\in M(\omega)}d(x,y)$
is lower semi-continuous  for any $x\in X$,
one has $\phi(n,\omega,x_s)\in U_{s(n)}$ and
\begin{align*}
\inf_{y\in K(\theta^n\omega)} d(\phi(n,\omega,x_s),y)&\le \liminf_{k\rightarrow +\infty} \inf_{y\in K(\theta^n\omega^{i_k})} d(\phi(n,\omega,x_{s}),y)\\
&\le  \liminf_{k\rightarrow +\infty} \inf_{y\in K(\theta^n\omega^{i_k})} \big( d(\phi(n,\omega,x_s),\phi(n,\omega^{i_k},x_{s^k}))+d(\phi(n,\omega^{i_k},x_{s^k}),y)\big)\\
&=0.
\end{align*}
Thus $\phi(n,\omega,x_s)\in U_{s(n)}\cap K(\theta^n \omega)$ for $n\in J$. This implies $(\omega,u)\in Y$. Hence $Y$ is a closed subset of $\Omega\times \{0,1\}^{\mathbb{N}}$.
This completes the proof of claim.
\end{proof}
Let $G_P$ be the set of all generic points of $(\Omega, \mathcal{F}, {P},\theta)$, that is,
$\omega\in G_P$ if and only if $P=\lim_{m\rightarrow +\infty} \frac{1}{m}\sum_{j=0}^{m-1}\delta_{\theta^j \omega}$ in the weak*-topology.
By Birkhoff pointwise ergodic Theorem, $P(G_P)=1$ as $(\Omega, \mathcal{F}, {P},\theta)$ is ergodic.

As $P(G_P\cap \Omega_0)=1$, we take $\omega_*\in G_P\cap \Omega_0$. By the above (2),  for any natural
number $m\ge M_{b,\omega_*}$, there is a subset $J_m(w_*)\subset
\{0,1,2,\cdots,m-1\}$ with $|J_m(\omega_*)|\ge bm$(positive density), and for any $s\in
\{1,2\}^{J_m(\omega_*)}$, there exists $x_s\in K(\omega_*)$ with
$\phi(j,\omega_*,x_s)\in U_{s(j)}\cap K(\theta^j \omega_*)$ for any
$j\in J_m(\omega_*)$.

For $m\ge M_{b,\omega_*}$. Let $v_m\in \{0,1\}^{\mathbb{N}_0}$ with $v_m(n)=1$ if and only if $n\in J_m(\omega_*)$, that is, $\{n\in \mathbb{N}_0:v_m(n)=1\}=J_m(\omega_*)$.
Clearly $(\omega_*,v_m)\in Y$.  Assume $\mu_m =\frac{1}{m}\sum_{j=0}^{m-1}\delta_{T^j(\omega_*,v_m)}$  for $m\ge M_{b,\omega_*}$. Then
each $\mu_m$ is a Borel probability on $Y$.
Let $\mu= \lim_{i\rightarrow +\infty}\mu_{m_i}$ be a limit point of $\{\mu_m\}$ in the weak*-topology. Clearly, $\mu$ is a $T$-invariant Borel probability on $Y$.

Let $\pi:Y\rightarrow \Omega$ be the projection of coordinate. Then $\mu\circ \pi^{-1}=P$. In fact, for any continuous function $f\in C(\Omega)$,
\begin{align*}
\int_\Omega f d\mu\circ \pi^{-1}&=\int_Y f\circ \pi d\mu=\lim_{i\rightarrow +\infty} \int_Y f\circ \pi d\mu_{m_i}\\
&=\lim_{i\rightarrow +\infty}\frac{1}{m_i}\sum_{j=0}^{m_i-1}f\circ \pi(T^j(\omega_*,v_m))\\
&=\lim_{i\rightarrow +\infty}\frac{1}{m_i}\sum_{j=0}^{m_i-1}f(\theta^j\omega_*)\\
&=\int_\Omega f d P
\end{align*}
the last equality comes from the fact $\omega_*$ is a generic point of $(\Omega, \mathcal{F}, {P},\theta)$.
Thus $\mu\circ \pi^{-1}=P$.

Let $[1]=\{(\omega,u)\in Y: u(0)=1 \}$. Then $[1]$ is a closed and open subset of $Y$. Note that $T^j(\omega,u)=(\theta^j\omega,\sigma^ju)$, one has
\begin{align*}
\mu([1])&=\lim_{i\rightarrow +\infty}\mu_{m_i}([1])=\lim_{i\rightarrow +\infty}\frac{1}{m_i}|\{j\in \{0,1,\cdots,m_i-1\}:(\sigma^jv_{m_i})(j)=1\}|\\
&=\lim_{i\rightarrow +\infty}\frac{1}{m_i}|\{j\in \{0,1,\cdots,m_i-1\}:v_{m_i}(j)=1\}|\\
&=\lim_{i\rightarrow +\infty}\frac{1}{m_i}|J_{m_i}(\omega_*)|\\
&\ge b.
\end{align*}

By the ergodic decomposition, $\mu\circ \pi^{-1}=P$ and the fact that $(\Omega, \mathcal{F}, {P},\theta)$ is ergodic, we know that
there exists a $T$-invariant ergodic Borel probablity measure $\nu$ on $Y$ with $\nu([1])\ge \mu([1])\ge b$ and $\nu\circ \pi^{-1}=P$.
Let $G_\nu$ be  the set of all generic points of $(Y,\mathcal{B}_Y,\nu,T)$, that is,
$(\omega,u)\in G_\nu$ if and only if $\nu=\lim_{m\rightarrow +\infty} \frac{1}{m}\sum_{j=0}^{m-1}\delta_{T^j(\omega,u)}$ in the weak*-topology.
Then $G_\nu$ is a Borel subset of $Y$, and $\nu(G_\nu)=1$ by Birkhoff pointwise ergodic Theorem.

Let $\Omega_1=\pi(G_\nu)$. Then $\Omega_1$ is a $P$-measurable set since it is the continuous image of a Borel set.
Thus $P(\Omega_1)=\nu\circ \pi^{-1}(\Omega_1)=\nu(\pi^{-1}(\Omega_1))=1$.
Now for $\omega \in \Omega_1$, there exists $u_\omega\in \{0,1\}^{\mathbb{N}_0}$ such that
$(\omega,u_\omega)\in G_\nu$. Let $J(\omega)=\{n\in \mathbb{N}_0:u_\omega(n)=1\}$.
As $(\omega,u_\omega)\in G_\nu$ and $[1]$ is closed and open, one has
\begin{align*}
\nu([1])&=\lim_{m\rightarrow +\infty} \frac{1}{m}\sum_{j=0}^{m-1}\delta_{T^j(\omega,u_\omega)}([1])\\
&=\lim_{m\rightarrow +\infty} \frac{1}{m}|\{j\in \{0,1,\cdots,m-1\}:(\sigma^ju_{\omega})(0)=1\}|\\
&=\lim_{m\rightarrow +\infty}\frac{1}{m}|\{j\in \{0,1,\cdots,m_i-1\}:u_\omega(j)=1\}|\\
&=\lim_{m\rightarrow +\infty}\frac{1}{m}|J(\omega)\cap \{0,1,\cdots,m-1\}|
\end{align*}
Thus the limit
$\lim_{m\rightarrow +\infty}\frac{1}{m}|J(\omega)\cap \{0,1,2,\cdots m\}|$ exists and is larger than or equal to $b$, as $\nu([1])\ge b$.
Note that $(\omega,u_\omega)\in Y$. By the definition of $Y$, one has for any $s\in \{1,2\}^{J(\omega)}$, there exists $x_s\in K(\omega)$ with
$\phi(j,\omega,x_s)\in U_{s(j)}\cap K(\theta^j \omega)$ for any
$j\in J(\omega)$. Summing up the above discussion, $\{U_1,U_2\}$ is a full Horseshoe
 for $(\phi,K)$ and completes the proof of the Theorem.
\hfill$\square$.

\section{Proof of Theorem \ref{MTH}} \label{Li-Yorke-chaos}

To prove this theorem, we also need the following result due to Mycielski (see \cite[Theorem 5.10]{Ak}
and \cite[Theorem 6.32]{AAG}).

\begin{lem}(Mycielski) \label{Myc} Let $Y$ be a perfect compact metric space and $C$ a symmetric dense
$G_\delta$ subset of $Y\times Y$. Then there exists a dense Mycielski subset $K\subseteq Y$ such that
$K\times K\subseteq C\cup \Delta_Y$, where $\Delta_Y=\{ (y,y):y\in Y\}$. \end{lem}

\smallskip

\noindent{\bf Proof of Theorem \ref{MTH}.} Assume that
$ h_{top}(\phi,K)>0$. By Proposition \ref{VP} there exists $\mu\in \mathcal{E}_P^K(\Omega\times X)$  such
that $h_\mu(\phi)>0$. Since $K$ is a Borel subset of $\Omega\times X$, $K$ is also Polish space and the Borel
$\sigma$-algebra $\mathcal{B}_K$ of $K$ is just $\{ A\cap K: A\in  \mathcal{F}\times \mathcal{B}_X\}$. Thus
$(K,\mathcal{B}_K,\mu, \Phi)$ is an  ergodic Polish system and $\pi_\Omega: (K,\mathcal{B}_K,\mu,
\Phi)\rightarrow (\Omega, \mathcal{F},P,\theta)$ is a factor map between two Polish systems. Recall that the
entropy of RDS $\phi$ with respect to $\mu$  is given by
$$h_\mu(\phi):=h_\mu(\Phi|\pi_\Omega)=h_\mu(\Phi|\pi_\Omega^{-1}(\mathcal{F})).$$

Let $\mathcal{B}_P$ be the completion  of  $\mathcal{F}$ with  respect to  $P$. Let $\mathcal{B}_\mu$ be the
completion  of  $\mathcal{B}_K$ with  respect to  $\mu$. Then $\pi_\Omega:(K,\mathcal{B}_\mu,\mu,
\Phi)\rightarrow (\Omega,  \mathcal{B}_P,P,\theta)$  is a factor map between  two Lebesgue  systems. Since
$\mathcal{B}_P=\mathcal{F}$ (mod $P$) and  $\mathcal{B}_\mu=\mathcal{B}_K$ (mod  $\mu$), it  is  clear  that
$$
h_\mu(\Phi|\pi_\Omega^{-1}(\mathcal{B}_P))=h_\mu(\Phi|\pi_\Omega^{-1}(\mathcal{F}))=h_\mu(\phi).
$$ Let
$\Pi_K:(\bar{K},\bar{\mathcal{B}_{\mu}},\bar{\mu}, \bar{\Phi})\rightarrow (K,\mathcal{B}_\mu,\mu, \Phi)$ be
the  natural extension of $(K,\mathcal{B}_\mu,\mu, \Phi)$. Let $$\Pi_{\Omega}: (\bar{\Omega},
\bar{\mathcal{B}_{P}},\bar{P},\bar{\theta})\rightarrow (\Omega, \mathcal{B}_P,P,\theta)$$ be the natural
extension of $(\Omega, \mathcal{B}_P,P,\theta)$. Define $\bar{\pi}:\bar{K}\rightarrow \bar{\Omega}$ by
$\bar{\pi}(((\omega_i,x_i))_{i\in \mathbb{Z}})=(\omega_i)_{i\in \mathbb{Z}}$ for $((\omega_i,x_i))_{i\in
\mathbb{Z}}\in \bar{K}$.  Then $\bar{\pi}:(\bar{K},\bar{\mathcal{B}_{\mu}},\bar{\mu},\bar{\Phi})\rightarrow
(\bar{\Omega},\bar{\mathcal{B}_{P}},\bar{P},\bar{\theta})$ be a factor map between two invertible Lebesgue
systems and the following diagram is  commutative: \begin{displaymath} \begin{CD}
(K,\mathcal{B}_\mu,\mu,\Phi)  @< \Pi_K<<  (\bar{K},\bar{\mathcal{B}_{\mu}},\bar{\mu},\bar{\Phi})\\ @VV\pi V
@VV\bar{\pi} V  \\ (\Omega,\mathcal{B}_P,P,\theta) @<  \Pi_\Omega<<
(\bar{\Omega},\bar{\mathcal{B}_{P}},\bar{P},\bar{\theta})\\ \end{CD} \end{displaymath}

We divide the remainder of the proof  into  four  steps.

\medskip \noindent{\it Step  1.}  Let $P_{\bar{\mu}}(\bar{\pi})$ be the relative Pinsker $\sigma$-algebra of
$(\bar{K},\bar{\mathcal{B}_{\mu}},\bar{\mu},\bar{\Phi})$ relative  to $\bar{\pi}$. Since
${\bar{\Phi}}^{-1}(P_{\bar{\mu}}(\bar{\pi}))=P_{\bar{\mu}}(\bar{\pi})$  and
$$\bar{\pi}^{-1}(\mathcal{\bar B}_{{P}})\subseteq P_{\bar{\mu}}(\bar{\pi})\subseteq \bar{\mathcal{B}_{\mu}},$$
there exist an invertible  Lebesgue system  $(Y,\mathcal{D},\nu,S)$ and  two factor maps
$$\pi_1:(\bar{K},\bar{\mathcal{B}_{\mu}},\bar{\mu},\bar{\Phi})\rightarrow (Y,\mathcal{D},\nu,S),\
\pi_2:(Y,\mathcal{D},\nu,S)\rightarrow  (\bar{\Omega},\bar{\mathcal{B}_{P}},\bar{P},\bar{\theta})$$ between
invertible  Lebesgue systems such that  $\pi_2\circ \pi_1=\bar{\pi}$  and
$\pi_1^{-1}(\mathcal{D})=P_{\bar{\mu}}(\bar{\pi})$ (mod $\bar{\mu}$). That is, $\pi_1$ is the Pinsker factor
of $(\bar{K},\bar{\mathcal{B}_{\mu}},\bar{\mu},\bar{\Phi})$ relative  to $\bar{\pi}$.

Let  $\bar{\mu}=\int_Y \bar{\mu}_y d \nu(y)$ be the disintegration of $\bar{\mu}$ relative to the  factor
$(Y,\mathcal{D},\nu,S)$. Let $$(\bar{K}\times \bar{K},\bar{\mathcal{B}_{\mu}}\times
\bar{\mathcal{B}_{\mu}},\bar{\mu} \times_Y\bar{\mu}, \bar{\Phi}\times \bar{\Phi})$$  be the product of
$(\bar{K},\bar{\mathcal{B}_{\mu}},\bar{\mu},\bar{\Phi})$ with itself relative to a  factor
$(Y,\mathcal{D},\nu,S)$. Recall that $$(\bar{\mu}\times_Y \bar{\mu})(B)=\int_{Y}(\bar{\mu}_y\times
\bar{\mu}_y)(B) \, d \nu(y), \ B\in \bar{\mathcal{B}_{\mu}}\times \bar{\mathcal{B}_{\mu}}.$$ Since
$(K,\mathcal{B}_{\mu},\mu,\Phi)$ is ergodic, we have that
$(\bar{K},\bar{\mathcal{B}_{\mu}},\bar{\mu},\bar{\Phi})$ is ergodic. By \eqref{na=en-1} in the proof of Theorem \ref{MTH1},
$h_{\bar{\mu}}(\bar{\Phi}|\bar{\pi})>0$.  Moreover by Lemma \ref{key-lem}, we have
\begin{enumerate} \item[{(a1).}]
$\bar{\mu}_y$ is non-atomic for $\nu$-a.e. $y\in Y$.

\item[{(a2).}] $(\bar{K}\times \bar{K},\bar{\mathcal{B}_{\mu}}\times \bar{\mathcal{B}_{\mu}},\bar{\mu}
    \times_Y\bar{\mu}, \bar{\Phi}\times \bar{\Phi})$ is ergodic.
\end{enumerate}

\medskip \noindent{\it Step 2.} We define a metric $\rho$ on $X^{\mathbb{Z}}$ as follows:
$$\rho((x_i^1)_{i\in  \mathbb{Z}},  (x_i^2)_{i\in  \mathbb{Z}})=\sum_{i\in  \mathbb{Z}} \frac{1}{2^{|i|}}
\frac{d(x_i^1,x_i^2)}{1+d(x_i^1,x_i^2)} \ \text{ for } (x^1_i)_{i\in  \mathbb{Z}},(x_i^2)_{i\in
\mathbb{Z}}\in X^{\mathbb{Z}}.$$

For $n\in  \mathbb{N}$, let $$A_n=\{ ((\omega^1_i,x^1_i))_{i\in  \mathbb{Z}}, ((\omega^2_i,x^2_i))_{i\in
\mathbb{Z}})\in  \bar{K}\times \bar{K} :\rho((x^1_i)_{i\in  \mathbb{Z}},(x_i^2)_{i\in
\mathbb{Z}})<\frac{1}{n}\}$$ and $$B_n=\bar{K}\times \bar{K}\setminus  A_n= \{ ((\omega^1_i,x^1_i))_{i\in
\mathbb{Z}}, ((\omega^2_i,x^2_i))_{i\in \mathbb{Z}})\in  \bar{K}\times \bar{K} :\rho((x^1_i)_{i\in
\mathbb{Z}},(x_i^2)_{i\in  \mathbb{Z}})\ge \frac{1}{n}\} .$$ Then $A_n,B_n$ are measurable subsets of
$\bar{K}\times \bar{K}$. Now we  claim:

\noindent{\bf Claim 1:} $\bar{\mu}\times_Y \bar{\mu}(\bigcup \limits_{n=1}^\infty B_n)=1$ and
$\bar{\mu}\times_Y \bar{\mu}(A_n)>0$ for each  $n\in  \mathbb{N}$.

\begin{proof}[Proof of  Claim 1]  Let $\Delta=\{(((\omega_i^1,x^1_i))_{i\in  \mathbb{Z}},
((\omega_i^2,x^2_i))_{i\in  \mathbb{Z}})\in \bar{K}\times \bar{K}:x^1_i=x_i^2 \text{ for each }i\in
\mathbb{Z}\}$. Then $$\bar{K}\times \bar{K}\setminus \Delta=\bigcup_{n=1}^\infty B_n.$$ Hence to show
$\bar{\mu}\times_Y \bar{\mu}(\bigcup \limits_{n=1}^\infty B_n)=1$, it is sufficient to show that
$\bar{\mu}\times  \bar{\mu}(\Delta)=0$. For $y\in Y$, $\pi_2(y)\in  \Omega$ and
$\bar\mu_y(\bar{\pi}^{-1}(\pi_2(y)))=1$.
Hence $$\bar{\mu}_y\times \bar{\mu}_y(\Delta)=\bar{\mu}_y\times
\bar{\mu}_y(\bar{\pi}^{-1}(\pi_2(y))\times \bar{\pi}^{-1}(\pi_2(y))\cap \Delta)=\bar{\mu}_y\times
\bar{\mu}_y(\Delta_{\bar{K}})$$ where $\Delta_{\bar{K}}=\{ (\bar{k},\bar{k}):\bar{k}\in \bar{K}\}$.

It is clear that when  $\bar{\mu}_y$ is non-atomic, i.e., $\bar{\mu}_y(\{  \bar{k}\})=0$ for $\bar{k}\in
\bar{K}$, we have  $$\bar{\mu}_y\times  \bar{\mu}_y(\Delta_{\bar{K}})=\int_{\bar{K}}\left(
\int_{\bar{K}}1_{\Delta_{\bar{K}}}(\bar{k}_1,\bar{k}_2) d\bar{\mu}_y(\bar{k}_2) \right)
d\bar{\mu}_y(\bar{k}_1)=\int_{\bar{K}}\bar{\mu}_y(\{\bar{k}_1\}) d \bar{\mu}_y(\bar{k}_1)=0
$$ where
$1_{\Delta_{\bar{K}}}$ is the characterization function  of  $\Delta_{\bar{K}}$. By (a1) in step 1, for
$\nu$-a.e. $y\in Y$, $\bar{\mu}_y$ is non-atomic. Thus for  $\nu$-a.e. $y\in Y$, $\bar{\mu}_y\times
\bar{\mu}_y(\Delta) = \bar{\mu}_y\times \bar{\mu}_y(\Delta_{\bar{K}})=0$.
Moreover,
\begin{align*}
\bar{\mu}\times_Y \bar{\mu}(\Delta)=\int_Y \bar{\mu}_y\times
\bar{\mu}_y(\Delta) d\nu(y)=0.
\end{align*} This
implies $\bar{\mu}\times_Y \bar{\mu}(\bigcup \limits_{n=1}^\infty B_n)=1$.

Next,  for  $n\in \mathbb{N}$, we are to show $\bar{\mu}\times_Y  \bar{\mu}(A_n)>0$. For $\bar{z}\in
X^{\mathbb{Z}}$, we define $$B_\rho(\bar{z},n)=\{ \bar{x}\in X^{\mathbb{Z}}:
\rho(\bar{x},\bar{z})<\frac{1}{2n}\}.$$ Since $(X,d)$ is a separable metric space,  $(X^{\mathbb{Z}},\rho)$
is also separable. Thus there exist $\{ \bar{z}_i\}_{i=1}^\infty \subseteq X^{\mathbb{Z}}$  such that
$\bigcup_{i=1}^\infty B_\rho(\bar{z}_i,n)=X^{\mathbb{Z}}$.  Put $$D_i=\{ ((\omega_i,x_i))_{i\in
\mathbb{Z}}\in \bar{K}: (x_i)_{i\in \mathbb{Z}}\in B_\rho(\bar{z}_i,n)\}.$$ Then $\bigcup_{i=1}^\infty
D_i=\bar{K}$ and there exists $r\in \mathbb{N}$ such that $\bar{\mu}(D_r)>0$. It is clear  that  $D_r\times
D_r\subseteq A_n$ and \begin{align*} \bar{\mu}\times_Y \bar{\mu}(A_n)&\ge \bar{\mu}\times_Y
\bar{\mu}(D_r\times  D_r)=\int_Y \bar{\mu}_y\times  \bar{\mu}_y(D_r\times D_r)d \nu(y)\\ &=\int_Y
(\bar{\mu}_y(D_r))^2 d  \nu(y)\ge (\int_Y \bar{\mu}_y(D_r) d \nu(y))^2=\bar{\mu}(D_r)^2>0. \end{align*} This
completes the proof of Claim 1. \end{proof}

\medskip \noindent{\it Step 3.} Let \begin{align*} \bar{\Omega}_0=\{& (\omega_i)_{i\in \mathbb{Z}}\in
\bar{\Omega}: \phi(\omega_i):X\rightarrow X
 \text{ is  continuous and }K_{\omega_i} \text{ is compact subset of }\\
 &\hskip1cm X \text{ for each } i\in \mathbb{Z}\}.
\end{align*}
 Since $\phi(\omega):X\rightarrow X$ is  continuous and $K_\omega$ is compact subset of $X$
 for $P$-a.e. $\omega\in \Omega$,
 the set $$\Omega_0:=\{  \omega\in \Omega: \phi(\omega):X\rightarrow X \text{ is  continuous and }K_{\omega}
 \text{ is compact subset of }X\}$$
 is full measure, that is $P(\Omega_0)=1$. Note that
 $\bar{\Omega}_0=\bigcap_{n\in \mathbb{Z}}\Pi_{n,\Omega}^{-1}(\Omega_0)$,
 we have $\bar{P}(\bar{\Omega}_0)=1$ since for each $n\in \mathbb{Z}$,
 $\bar{P}(\Pi_{n,\Omega}^{-1}\Omega_0)=
 P(\Omega_0)=1$.  Moreover
 \begin{equation}\label{mes=1}
 \nu(\pi_2^{-1}(\bar{\Omega}_0))=\bar{P}(\bar{\Omega}_0)=1.
 \end{equation}

By the step 2, there exists $r\in \mathbb{N}$ such that $\bar{\mu}\times_Y\bar{\mu}(B_r)>0$. Let
\begin{align*}
 W=\{ &(\bar{k}_1,\bar{k}_2)\in \bar{K}\times \bar{K}: \lim_{\ell\rightarrow \infty}
 \frac{1}{\ell}\sum_{i=0}^{\ell-1} 1_{A_n}((\bar{\Phi}\times
 \bar{\Phi})^i(\bar{k}_1,\bar{k}_2))=\bar{\mu}\times_Y \bar{\mu}(A_n)>0  \\
&\text{ for each }n\in \mathbb{N}\text{ and } \lim_{\ell\rightarrow \infty} \frac{1}{\ell}\sum_{i=0}^{\ell-1}
1_{B_r}((\bar{\Phi}\times \bar{\Phi})^i(\bar{k}_1,\bar{k}_2))=\bar{\mu}\times_Y \bar{\mu}(B_r)>0 \}
\end{align*} where $1_E$ is the characterization function of $E\subseteq \bar{K}\times \bar{K}$. Then,  $W$ is
a Borel subset of $\bar{K}\times \bar{K}$ and by the Birkhoff ergodic theorem $\bar{\mu}\times_Y
\bar{\mu}(W)=1$ since $\bar{\mu}\times_Y \bar{\mu}$ is ergodic (see (a2) in Step 1). Let $$Y_0=\{ y\in  Y:
\bar{\mu}_y\times \bar{\mu}_y(W)=1, \mu_y \text{ is non-atomic}\}\cap \pi_2^{-1}(\bar{\Omega}_0).$$ Note that
$1=\bar{\mu}\times_Y \bar{\mu}(W)=\int_Y \bar{\mu}_y\times \bar{\mu}_y(W) d \nu(y)$ since
$\bar{\mu}_y\times \bar{\mu}_y(W)=1$ for $\nu$-a.e. $y\in Y$.  Combing this with \eqref{mes=1} and  (a1) in
Step 1, we have  $\nu(Y_0)=1$.

For  $y\in Y$, let $(\omega_i(y))_{i\in\mathbb{Z}}:=\pi_2(y)\in \bar{\Omega}$ and
$$\bar{K}_y=\{(x_i)_{i\in
\mathbb{Z}}\in \prod_{i\in  \mathbb{Z}}K_{\omega_i(y)}: \phi(\omega_i(y))(x_i)=x_{i+1} \text{ for each
}i\in\mathbb{Z}\}.$$ For simplicity,  we identify $(\Omega\times  X)^{\mathbb{Z}}$ with
$\Omega^{\mathbb{Z}}\times X^{\mathbb{Z}}$  as follow:
$$((\omega_i,x_i))_{i\in \mathbb{Z}}\in (\Omega\times
X)^{\mathbb{Z}}\simeq ((\omega_i)_{i\in \mathbb{Z}}, (x_i)_{i\in \mathbb{Z}})\in \Omega^{\mathbb{Z}}\times
X^{\mathbb{Z}}.
$$ Clearly  $$\pi_1^{-1}(y)\subseteq \bar{\pi}^{-1}(\pi_2(y))\subseteq \{\pi_2(y)\}\times
\bar{K}_y.$$ Put $\kappa=\frac{1}{8r}$, we have the following claim.

\medskip \noindent{\it Claim 2.} For $y\in Y_0$, $\bar{K}_{y}$ is  a compact subset of
$(X^{\mathbb{Z}},\rho)$ and there exists a Mycielski subset $D_y$ of  $\bar{K}_y$ such that for each pair
$((x^1_i)_{i\in  \mathbb{Z}},(x_i^2)_{i\in \mathbb{Z}})$ of distinct points in  $D_y$, we have
\begin{align}\label{sat-LY}
\begin{aligned} &\limsup \limits_{n\rightarrow +\infty}
\rho((\phi(n, {\omega_i(y)}, x_i^1))_{i\in \mathbb{Z}}, (\phi(n, {\omega_i(y)}, x_i^2))_{i\in  \mathbb{Z}})\ge 4\kappa
\text{ and }\\ &\liminf \limits_{n\rightarrow +\infty}  \rho((\phi(n, {\omega_i(y)}, x_i^1))_{i\in \mathbb{Z}},
(\phi(n, {\omega_i(y)}, x_i^2))_{i\in  \mathbb{Z}})=0
\end{aligned} \end{align} where
  $\omega\in \Omega$.

\begin{proof}[Proof  of Claim 2] Let $y\in Y_0$. First by the choice of $Y_0$,
$\phi(\omega_i(y)):(X,d)\rightarrow (X,d)$ is continuous and
$K_{\omega_i(y)}$ is  compact subset of $(X,d)$ for each
$i\in\mathbb{Z}$, where $(\omega_i(y))_{i\in \mathbb{Z}}:=\pi_2(y)$.
Thus $\prod_{i\in \mathbb{Z}}K_{\omega_i(y)}$ is  a compact subset
of  $(X^{\mathbb{Z}},\rho)$. Since  $\bar{K}_y$ is a closed subset
of $\prod_{i\in  \mathbb{Z}}K_{\omega_i(y)}$, $\bar{K}_y$ is  also a
compact subset of $(X^{\mathbb{Z}},\rho)$.   We define
\begin{align*} E_y&=\{ (x_i)_{i\in \mathbb{Z}}\in \bar K_y: \bar{\mu}_y(\{
\pi_2(y)\} \times  (U\cap {\bar K_y}))>0 \text{ for any open
neighborhood $U$ of }
\\ &\hskip1cm (x_i)_{i\in \mathbb{Z}}  \text{  in }
(X^{\mathbb{Z}},\rho)\}
\end{align*}
Since $\bar{\mu}_y(\pi_1^{-1}(y))=1$ and $\pi_1^{-1}(y)\subseteq
\{\pi_2(y)\}\times \bar{K}_y$,  we have
$\bar{\mu}_y(\{\pi_2(y)\}\times \bar{K}_y)=1$. Moreover,  $E_y$ is a
closed subset of $\bar{K}_y$ since $\bar{K}_y$ is  compact. Thus
$E_y$ is compact and $\bar{\mu}_y(\{\pi_2(y)\}\times E_y)=1$. Note
that $\bar {\mu}_y$ is  non-atomic, $E_y$ is a perfect set. Hence,
$(E_y,\rho)$ is a perfect compact metric space.

For  $n,m\in\mathbb{N}$, put
\begin{align*} D_n(y)=\{&((x_i^1)_{i\in \mathbb{Z}},(x_i^2)_{i\in
\mathbb{Z}})\in E_y\times E_y: \text{ there exists $k\ge n$ }\\ &\hskip0.5cm \text{ such that
}\rho\Big((\phi(k, {\omega_i(y)}, x_i^1))_{i\in \mathbb{Z}}, (\phi(k, {\omega_i(y)}, x_i^2))_{i\in
\mathbb{Z}}\Big)>4\kappa\}
\end{align*}  and
\begin{align*} P_{n,m}(y)=\{&((x_i^1)_{i\in
\mathbb{Z}},(x_i^2)_{i\in \mathbb{Z}})\in E_y\times E_y: \text{
there exists $p\ge n$ }\\ &\hskip0.5cm \text{ such that
}\rho\left((\phi(p, {\omega_i(y)}, x_i^1))_{i\in \mathbb{Z}},
(\phi(p, {\omega_i(y)}, (x_i^2))_{i\in \mathbb{Z}}\right)
<\frac{1}{m} \}. \end{align*}  Since  $(x_i)_{i\in
\mathbb{Z}}\mapsto (\phi(n, {\omega_i(y)}, x_i))_{i\in \mathbb{Z}}$
is  a continuous map from $(X^{\mathbb{Z}},\rho)$ to itself,
$D_n(y)$ and $P_{n,m}(y)$  are both open  subsets of $(E_y\times
E_y,\rho\times  \rho)$. Let $$C(y)=\left( \bigcap_{n=1}^\infty
D_n(y)\right)\cap \left(\bigcap_{m=1}^\infty (\bigcap_{n=1}^\infty
P_{n,m}(y))\right).$$ Then $C(y)$ is a $G_\delta$ subset of
$E_{y}\times  E_y$ and for any $((x^1_i)_{i\in
\mathbb{Z}},(x_i^2)_{i\in \mathbb{Z}})\in  C(y)$, $((x^1_i)_{i\in
\mathbb{Z}},(x_i^2)_{i\in \mathbb{Z}})$ satisfies \eqref{sat-LY}. We
claim $C(y)$ is  dense  in $E_{y}\times  E_y$. If this is not true,
then there exist open  subsets $U_1$, $U_2$ of  $X^{\mathbb{Z}}$
such that $(U_1\times U_2)\cap (E_y\times E_y)\neq \emptyset$ and
$(U_1\times U_2)\cap C(y)=\emptyset$. Since
$\bar{\mu}_y(\{\pi_2(y)\}\times E_y)=1$, one has \begin{align*}
&\bar{\mu}_y\times
\bar{\mu}_y((\pi_2(y)\times (U_1\cap E_y))\times (\pi_2(y)\times (U_2\cap E_y)))\\
=&\bar{\mu_y}(\pi_2(y)\times (U_1\cap E_y))\bar{\mu_y}(\pi_2(y)\times (U_2\cap E_y))\\
=&\bar{\mu_y}(\pi_2(y)\times U_1)\bar{\mu_y}(\pi_2(y)\times U_2)>0.
\end{align*}
Combing this with the  fact $\bar{\mu}_y\times\bar{\mu}_y(W)=1$, we
have \begin{align*} &\bar{\mu}_y\times
\bar{\mu}_y(\left((\pi_2(y)\times (U_1\cap E_y))\times (\pi_2(y)\times (U_2\cap  E_y))\right)\cap W)\\
=&\bar{\mu}_y\times \bar{\mu}_y((\pi_2(y)\times U_1)\times (\pi_2(y)\times U_2))>0. \end{align*}
Particularly, $((\pi_2(y)\times (U_1\cap E_y))\times (\pi_2(y)\times (U_2\cap  E_y))\cap W\neq \emptyset$.
Hence there exist $(z^1_i)_{i\in \mathbb{Z}}\in U_1\cap  E_y$ and $(z^2_i)_{i\in \mathbb{Z}}\in U_2\cap  E_y$
such that $(((\omega_i(y),z_i^1))_{i\in\mathbb{Z}}, ((\omega_i(y),z_i^2))_{i\in \mathbb{Z}})\in  W$.

Let  $n,m\in \mathbb{N}$. Since $((\omega_i(y),z_i^1))_{i\in\mathbb{Z}}, ((\omega_i(y),z_i^2))_{i\in
\mathbb{Z}})\in  W$, we have  $$\lim_{\ell\rightarrow \infty} \frac{1}{\ell}\sum_{i=0}^{\ell-1}
1_{B_r}((\bar{\Phi}\times \bar{\Psi})^i(((\omega_i(y),z_i^1))_{i\in\mathbb{Z}}, ((\omega_i(y),z_i^2))_{i\in
\mathbb{Z}})=\bar{\mu}\times_Y \bar{\mu}(B_r)>0$$
Hence there exists $k\ge n$ such that $(\bar{\Phi}\times
\bar{\Phi})^k(((\omega_i(y),z_i^1))_{i\in\mathbb{Z}}, ((\omega_i(y),z_i^2))_{i\in \mathbb{Z}})\in B_r$. Note
that
$$ \bar{\Phi}^k(((\omega_i(y),z_i^j))_{i\in\mathbb{Z}}= ((\theta^k\omega_i(y),
\phi(k, {\omega_i(y)}, z_i^j)))_{i\in \mathbb{Z}}$$ for $j=1,2$. Thus,  we have  $$\left( ((\theta^k\omega_i(y),
\phi(k, {\omega_i(y)}, z_i^j)))_{i\in \mathbb{Z}},((\theta^k\omega_i(y), \phi(k, {\omega_i(y)}, z_i^j)))_{i\in
\mathbb{Z}}\right)\in B_r.$$ Thus $\rho((\phi(k, {\omega_i(y)}, z_i^j))_{i\in \mathbb{Z}},
(\phi(k, {\omega_i(y)}, z_i^j))_{i\in \mathbb{Z}})\ge \frac{1}{r}> 4\kappa$. Hence \begin{equation}\label{dny}
((z^1_i)_{i\in \mathbb{Z}},(z^2_i)_{i\in \mathbb{Z}})\in D_n(y). \end{equation}

Since $((\omega_i(y),z_i^1))_{i\in\mathbb{Z}}, ((\omega_i(y),z_i^2))_{i\in \mathbb{Z}})\in  W$, we have
$$\lim_{\ell\rightarrow \infty} \frac{1}{\ell}\sum_{i=0}^{\ell-1} 1_{A_m}((\bar{\Phi}\times
\bar{\Phi})^i(\bar{k}_1,\bar{k}_2))=\bar{\mu}\times_Y \bar{\mu}(A_m)>0.$$
Thus there exits $p\ge n$ such that
$(\bar{\Phi}\times \bar{\Phi})^{p}(((\omega_i(y),z_i^1))_{i\in\mathbb{Z}}, ((\omega_i(y),z_i^2))_{i\in
\mathbb{Z}})\in A_m$. Note that
$$ \bar{\Phi}^p(((\omega_i(y),z_i^j))_{i\in\mathbb{Z}}= ((\theta^p\omega_i(y),
\phi(p, {\omega_i(y)}, z_i^j)))_{i\in \mathbb{Z}}$$ for $j=1,2$.  We have $$\left( ((\theta^p\omega_i(y),
\phi(p, {\omega_i(y)}, z_i^j)))_{i\in \mathbb{Z}},((\theta^p\omega_i(y), \phi(p, {\omega_i(y)}, z_i^j)))_{i\in
\mathbb{Z}}\right)\in A_m.$$ Thus $\rho((\phi(p, {\omega_i(y)}, z_i^j))_{i\in \mathbb{Z}},
(\phi(p,{\omega_i(y)}, z_i^j))_{i\in \mathbb{Z}})<\frac{1}{m}$. Hence \begin{equation}\label{pnmy} ((z^1_i)_{i\in
\mathbb{Z}},(z^2_i)_{i\in \mathbb{Z}})\in P_{n,m}(y). \end{equation} Since \eqref{dny} and  \eqref{pnmy} are
true for any $n,m\in \mathbb{N}$, we have  $((z^1_i)_{i\in \mathbb{Z}},(z^2_i)_{i\in \mathbb{Z}})\in C(y)$.
Therefore $((z^1_i)_{i\in \mathbb{Z}},(z^2_i)_{i\in \mathbb{Z}})\in (U_1\times U_2)\cap C(y)$, a
contradiction with $(U_1\times U_2)\cap C(y)=\emptyset$. This shows that $C(y)$ is a dense  $G_\delta$ subset
of $E_{y}\times  E_y$.

\medskip Using Lemma  \ref{Myc} for  $(C(y), E_y)$, there  exists a dense Mycielski subset $D_y\subseteq E_y$
such that $D_y\times D_y\subseteq C\cup \Delta_{E_y}$, where $\Delta_{E_y}=\{ (\bar{x},\bar{x}):\bar{x}\in
E_y\}$. Clearly, $D_y\subseteq E_y\subseteq \bar{K}_y$  and if $((x^1_i)_{i\in  \mathbb{Z}},(x_i^2)_{i\in
\mathbb{Z}})$ is a pair of distinct points in  $D_y$, then $((x^1_i)_{i\in  \mathbb{Z}},(x_i^2)_{i\in
\mathbb{Z}})\in  C(y)$ and hence $((x^1_i)_{i\in  \mathbb{Z}},(x_i^2)_{i\in \mathbb{Z}})$ satisfies
\eqref{sat-LY}. This completes the proof of Claim  2. \end{proof}

\medskip \noindent{\it Step  4.} For $P$-a.e. $\omega\in \Omega$ there exists a Mycieski chaotic set
$S_\omega\subset K_\omega$ for $(\omega,\phi)$.

\medskip Note that $$\Pi_\Omega\circ \pi_2:(Y,\mathcal{D},\nu,S)\rightarrow
(\Omega,\mathcal{F},\mathcal{B}_P,\theta)$$ is a factor map between two Lebesgue systems. Since $\nu(Y_0)=1$,
there exists $\Omega_1\in  \mathcal{B}_P$ satisfying $P(\Omega_1)=1$ and $(\Pi_\Omega\circ
\pi_2)^{-1}(\omega)\cap Y_0\neq \emptyset$ for  each $\omega\in \Omega_1$.

Let $\omega\in \Omega_1$. Then there exists $y\in Y_0$ such that $\Pi_\Omega\circ \pi_2(y)=\omega$. Since
$\pi_2(y)=(\omega_i(y))_{i\in \mathbb{Z}}$, we have  $\omega=\omega_0(y)$. By Claim  2,  we can find a
Mycielski subset $D_y$ of  $\bar{K}_y$ such that for each pair $((x^1_i)_{i\in  \mathbb{Z}},(x_i^2)_{i\in
\mathbb{Z}})$ of distinct points in  $D_y$, we have
\begin{align}\label{cha-eq} \begin{aligned} &\limsup
\limits_{n\rightarrow +\infty}  \rho((\phi(n,{\omega_i(y)}, x_i^1))_{i\in \mathbb{Z}},
(\phi(n,{\omega_i(y)}, x_i^2))_{i\in  \mathbb{Z}})\ge 4\kappa \text{ and }\\ &\liminf \limits_{n\rightarrow
+\infty}  \rho((\phi(n,{\omega_i(y)}, x_i^1))_{i\in \mathbb{Z}}, (\phi(n,{\omega_i(y)}, x_i^2))_{i\in
\mathbb{Z}})=0.
\end{aligned} \end{align}

Let $\eta:\bar{K}_y\rightarrow K_{\omega}$ be the natural projection of coordinate with $\eta((x_i)_{i\in
\mathbb{Z}})=x_0$ for $(x_i)_{i\in \mathbb{Z}}\in  X^{\mathbb{Z}}$. Put $S_\omega=\eta(D_y)$. Then
$S_\omega\subseteq K_{\omega}$. In the following we show  that $S_\omega$ is a Mycieski chaotic set for
$(\omega,f)$. Firstly we claim the map $\eta: D_y\rightarrow S_\omega$ is injective. If this is not true,
then there  exist two distinct points $(x_i^1)_{i\in  \mathbb{Z}},(x_i^2)_{i\in \mathbb{Z}}$ in $D_y$ such
that $\eta((x_i^1)_{i\in \mathbb{Z}})=\eta((x_i^2)_{i\in \mathbb{Z}})$,  i.e., $x_0^1=x_0^2$. Since
$(x_i^1)_{i\in  \mathbb{Z}},(x_i^2)_{i\in \mathbb{Z}}\in \bar{K}_y$, we  have
\begin{align}\label{eq-n}
x_i^1=\phi(i, {\omega_0(y)}, x_0^1)=\phi(i, {\omega_0(y)}, x_0^2)=x_i^2 \end{align} for $i\in \mathbb{N}$.  Now for
$n\in \mathbb{N}$, we have $$(\phi(n,{\omega_i(y)}, x_i^1))_{i\in \mathbb{Z}}=(x^1_{i+n})_{i\in \mathbb{Z}}
\text{ and } (\phi(n,{\omega_i(y)}, x_i^2))_{i\in \mathbb{Z}}=(x^2_{i+n})_{i\in \mathbb{Z}}. $$ Thus
$$\lim_{n\rightarrow +\infty} \rho((\phi(n,{\omega_i(y)}, x_i^1))_{i\in \mathbb{Z}},
(\phi(n,{\omega_i(y)}, x_i^2))_{i\in \mathbb{Z}})=\lim_{n\rightarrow +\infty}\rho((x^1_{i+n})_{i\in
\mathbb{Z}},(x^2_{i+n})_{i\in \mathbb{Z}} )=0$$ where the last equality comes from \eqref{eq-n}, a
contradiction with $$\lim_{n\rightarrow +\infty} \rho((\phi(n,{\omega_i(y)}, x_i^1))_{i\in \mathbb{Z}},
(\phi(n,{\omega_i(y)}, x_i^2))_{i\in \mathbb{Z}})\ge \kappa$$ since $(x_i^1)_{i\in  \mathbb{Z}},(x_i^2)_{i\in
\mathbb{Z}}$ are distinct points  in $D_y$. This  shows that $\eta: D_y\rightarrow S_\omega$ is injective.

Since $D_y$ is a Mycielski set, $D_y=\bigcup_{j\in \mathbb{N}}C_j$  where each $C_j$ is cantor set. Since
$\eta:(C_j,\rho)\rightarrow (\eta(C_j),d)$ is one to one surjective continuous map and $C_j$ is compact
subset of $(X^{\mathbb{Z}},\rho)$, $\eta:C_j\rightarrow \eta(C_j)$ is homeomorphism. Thus $\eta(C_j)$ is a
cantor set. Hence $S_\omega=\bigcup_{j\in \mathbb{N}} \eta(C_j)$ is also a mycielski set of $K_\omega$.

Finally, we prove that $S_\omega$  is a chaotic set for $(\omega,\phi)$. Let $(x^1,x^2)$ is a pair of distinct
points in $S_\omega$. Then there exist $(x_i^1)_{i\in \mathbb{Z}}, (x_i^2)_{i\in \mathbb{Z}}\in D_y$ such
that  $x_0^1=x^1$ and  $x_0^2=x^2$. On the  one hand, by \eqref{cha-eq} \begin{align*} &\hskip0.5cm \liminf
\limits_{n\rightarrow +\infty}  d(\phi(n, {\omega}, x^1), \phi(n, {\omega}, x^2))=\liminf \limits_{n\rightarrow
+\infty}  d((\phi(n, {\omega_0(y)}, x_0^1), \phi(n, {\omega_0(y)}, x_0^2))\\ &\le \liminf \limits_{n\rightarrow
+\infty}  \rho((\phi(n, {\omega_i(y)}, x_i^1))_{i\in \mathbb{Z}}, (\phi(n, {\omega_i(y)}, x_i^2))_{i\in
\mathbb{Z}})=0. \end{align*} Thus $\liminf \limits_{n\rightarrow +\infty}  d(f_{\omega}^n(x^1),
f_{\omega}^n(x^2))=0$.

On the other hand,  we take $L\in \mathbb{N}$ such that $\sum_{i\in \mathbb{Z},|i|\ge L}
\frac{1}{2^i}<\frac{1}{2}\kappa$. By \eqref{cha-eq} $$\limsup \limits_{n\rightarrow +\infty}
\rho((\phi(n, {\omega_i(y)}, x_i^1))_{i\in \mathbb{Z}}, (\phi(n, {\omega_i(y)}, x_i^2))_{i\in  \mathbb{Z}})\ge
4\kappa.$$ Hence there exist natural numbers $n_j,j\in \mathbb{N}$ such that $L<n_1<n_2<\cdots$  and
$$\rho((\phi(n_j, {\omega_i(y)}, x_i^1))_{i\in \mathbb{Z}}, (\phi(n_j, {\omega_i(y)}, x_i^2))_{i\in
\mathbb{Z}})>\frac{7}{2}\kappa.$$

Since $(x_i^1)_{i\in \mathbb{Z}}, (x_i^2)_{i\in \mathbb{Z}}\in \bar{K}_y$, we have that
$\phi(\omega_i(y))x_i^1=x_{i+1}^1$ and $\phi(\omega_i(y))x_i^2=x_{i+1}^2$ for $i\in \mathbb{Z}$. Moreover for each
$j\in \mathbb{N}$, $(\phi(n_j,{\omega_i(y)}, x_i^1))_{i\in \mathbb{Z}}=(x_{i+n_j}^1)_{i\in \mathbb{Z}}$ and
$(\phi(n_j, {\omega_i(y)}, x_i^2))_{i\in \mathbb{Z}}=(x_{i+n_j}^2)_{i\in \mathbb{Z}}$. Thus
$\rho((x_{i+n_j})_{i\in \mathbb{Z}}, (x_{i+n_j}^2)_{i\in  \mathbb{Z}})>\frac{7}{2}\kappa$ for each $j\in
\mathbb{N}$.
 Then
$$\sum_{i,|i|\le L}
\frac{1}{2^|i|}\frac{d(x_{i+n_j}^1,x_{i+n_j}^2)}{1+d(x_{i+n_j}^1,x_{i+n_j}^2)}>\frac{7}{2}\kappa-\sum_{i\in
\mathbb{Z},|i|\ge L} \frac{1}{2^i} \ge 3\kappa$$ for $j\in \mathbb{N}$. Thus for each $j\in \mathbb{N}$,
there exists $i_j\in \mathbb{Z}$ with $|i_j|\le L$ such that
$$\frac{d(x_{i_j+n_j}^1,x_{i_j+n_j}^2)}{1+d(x_{i+n_j}^1,x_{i_j+n_j}^2)}\ge \kappa,$$ which implies
$d(x_{i_j+n_j}^1,x_{i_j+n_j}^2)\ge \kappa$. Let $m_j=i_j+n_j>0$ for $j\in \mathbb{N}$. Then
$d(x_{m_j}^1,x_{m_j}^2)\ge \kappa$ and $\lim \limits_{j\rightarrow +\infty} m_j=+\infty$. Since
$$x_{m_j}^1=\phi(m_j, {\omega_0(y)}, x_0^1)=\phi(m_j, {\omega}, x^1) \text{ and }
x_{m_j}^2=\phi(m_j, {\omega_0(y)}, x_0^2)=\phi(m_j, {\omega}, x^2) \text{ }\forall j\in \mathbb{N}$$ we have that
$d(\phi(, m_j, \omega, x^1),\phi(m_j, \omega, x^2))=d(x_{m_j}^1,x_{m_j}^2)\ge \kappa$, $j\in \mathbb{N}$ and hence
$$\limsup \limits_{n\rightarrow +\infty}d(\phi(n, \omega, x^1),\phi(n, \omega, x^2))\ge \limsup \limits_{j\rightarrow
+\infty} d(\phi(m_j, \omega, x^1),\phi(m_j, \omega, x^2))\ge
\kappa.$$ Summerizing the above , we have that $$\liminf
\limits_{n\rightarrow +\infty}  d(\phi(n, {\omega}, x^1), \phi(n,
{\omega}, x^2))=0 \text{ and }\limsup \limits_{n\rightarrow
+\infty}d(\phi(n, \omega, x^1),\phi(n, \omega, x^2))\ge \kappa$$ for
a pair $(x^1,x^2)$  of distinct points in $S_\omega$. This shows
that $S_\omega$  is a chaotic set for $(\omega,f)$ and  completes
the proof of the theorem.
 \hfill$\square$

\appendix

\section{Proof of Basic Lemmas and Propositions}
In this appendix, we give the proof  of Lemma \ref{na=zero}, and several basic facts and their proofs about the ergodic decomposition, which we used in previous sections.

\smallskip

\noindent {\bf Proof of Lemma \ref{na=zero}} For $n\in \mathbb{Z}$, let
$\Pi_{n,X}:\bar{X}\rightarrow X$ with $\Pi_{n,X}((x_i)_{i\in \mathbb{Z}})=x_n$ for $\bar{x}=(x_i)_{i\in
\mathbb{Z}}\in X^{\mathbb{Z}}$.
Put
$$\bar{\mathcal{B}}_n=\Pi_{n,X}^{-1}(\mathcal{B}_X)$$ and
$\mathcal{D}_{\bar{X}}=\bigcup_{i\in \mathbb{Z}}\bar{\mathcal{B}}_i$. Clearly, \begin{align}\label{increase}
\bar{\mathcal{B}}_i\supseteq \bar{\mathcal{B}}_{i+1} \end{align} for each $i\in \mathbb{Z}$. Then
$\mathcal{D}_{\bar{X}}$ is a sub-algebra of $\bar{\mathcal{B}}$ and the $\sigma$-algebra generated by
$\mathcal{D}_{\bar{X}}$ is equal to $\bar{\mathcal{B}}$ (mod $\bar{\mu}$). Thus for any $A\in
\bar{\mathcal{B}}$ and $\epsilon>0$, there exists $A_\epsilon\in \mathcal{D}_{\bar{X}}$ such that
$\bar{\mu}(A\Delta A_\epsilon)<\epsilon$, where $A\Delta A_\epsilon=(A\setminus A_\epsilon)\cup
(A_\epsilon\setminus A)$ (see \cite[Theorem 0.7]{W}).

Let $\alpha=\{A_1,A_2,\cdots, A_k\}\in \mathcal{P}_{\bar{X}}$ with $k\ge 2$. For $m\in \mathbb{N}$,
there exists $\delta=\delta(k,m)>0$ such that for any $\beta=\{B_1,B_2,\cdots,B_k\}\in
\mathcal{P}_{\bar{X}}$, if $$\bar{\mu}(\alpha\Delta \beta):=\sum_{i=1}^k \bar{\mu}(A_i\Delta B_k)<\delta$$
then $H_\mu(\alpha|\beta)<\frac{1}{m}$ (see for example \cite[Lemma 4.15]{W}).

For  $i=1,2,\cdots,k-1$, we take $A_i'\in \mathcal{D}_{\bar{X}}$ with $\bar{\mu}(A_i\Delta
A_i')<\frac{\delta}{k^3}$. Let $A_k'=X\setminus \bigcup_{j=1}^{k-1}A_j'$. Then $A_k'\in
\mathcal{D}_{\bar{X}}$ and $$\bar{\mu}(A_k\Delta A_k')\le \mu(\bigcup_{j=1}^{k-1}A_j\Delta A_j')\le
\sum_{j=1}^{k-1} \mu(A_j\Delta A_j')<\frac{\delta}{k^2}$$ where the first inequality comes from the fact
$$A_k\Delta A_k'=(X\setminus \bigcup_{j=1}^{k-1}A_j)\Delta (X\setminus \bigcup_{j=1}^{k-1}A_j')\subseteq
\bigcup_{j=1}^{k-1} (A_j\Delta A_j').$$

Then let $C_1=A_1'$ and $$C_i=A_i'\setminus \bigcup_{j=1}^{i-1}C_j=A_i'\setminus \bigcup_{j=1}^{i-1}A_j'$$
for $i=2,\cdots,k$.  Clearly,  $C_i\in \mathcal{D}_{\bar{X}}$ for each $i\in \{1,2,\cdots,k\}$ and
$\gamma:=\{C_1,C_2,\cdots,C_k\}\in \mathcal{P}_{\bar{X}}$.

Since $\mathcal{D}_{\bar{X}}=\bigcup_{i\in \mathbb{Z}}\bar{\mathcal{B}}_i$, by \eqref{increase} we have
$i_*\in \mathbb{N}$ such  that $C_i\in  \bar{\mathcal{B}}_{-i_*}$ for each $i\in \{1,2,\cdots,k\}$. Thus,
there  exists  $D_i\in \mathcal{B}$ such that $\Pi_{-i_*,X}^{-1}(D_i)=C_i$ for $i=1,2,\cdots,k$. Let $\tau=\{
D_1,D_2,\cdots,D_k\}$. Then  $\tau\in \mathcal{P}_X$ and
$$\gamma=\Pi_{-i_*,X}^{-1}\tau=\bar{T}^{-i_*}\Pi_{0,X}^{-1}\tau=\bar{T}^{-i_*}\Pi_X^{-1}\tau.$$

For $i\in  \{1,2,\cdots,k\}$, we  have
\begin{align*} A_i\Delta C_i&=(A_i\setminus C_i)\cup (C_i\setminus
A_i)\subseteq (A_i\setminus (A_i'\setminus  \bigcup_{j=1}^{i-1}A_j'))\cup (A_i'\setminus A_i)\\
&=(A_i\setminus  A_i')\cup (A_i\cap \bigcup_{j=1}^{i-1}A_j')  \cup  (A_i'\setminus A_i)=(A_i\Delta A_i') \cup
(\bigcup_{j=1}^{i-1}  A_i\cap  A_j')\\ &\subseteq (A_i\Delta A_i') \cup (\bigcup_{j=1}^{i-1}  (A_i\cap  A_j)
\cup (A_i\cap (A_j'\setminus A_j)) )=(A_i\Delta A_i') \cup (\bigcup_{j=1}^{i-1}  A_i\cap (A_j'\setminus A_j)
\\ &\subseteq \bigcup_{j=1}^i (A_j\Delta A_j').
\end{align*}
Moreover,  $\bar{\mu}(A_i\Delta C_i)\le
\sum_{j=1}^i\mu(A_j\Delta A_j')<\frac{\delta}{k}$. Thus $\bar{\mu}(\alpha \Delta \gamma)=\sum_{i=1}^k
\bar{\mu}(A_i\Delta C_i)<\delta$. By the choice of $\delta$, we know
$H_{\bar{\mu}}(\alpha|\gamma)<\frac{1}{m}$.

Now for  $n\in \mathbb{N}$, we have
\begin{align*} &\hskip0.5cm
H_{\bar{\mu}}(\bigvee_{i=0}^{n-1}{\bar{T}}^{-i}\alpha|\Pi_X^{-1}(\mathcal{B}))=
H_{\bar{\mu}}(\bigvee_{i=0}^{n-1}{\bar{T}}^{-i}\alpha|\bar{\mathcal{B}}_0)=
H_{\bar{\mu}}(\bigvee_{i=0}^{n-1}{\bar{T}}^{-(i+i_*)}\alpha|{\bar{T}}^{-i_*}\bar{\mathcal{B}}_0)\\ &\le
H_{\bar{\mu}}(\bigvee_{i=0}^{n-1}{\bar{T}}^{-(i+i_*)}\alpha|{\bar{T}}^{-i_*}\Pi_X^{-1}(\bigvee_{i=0}^{n-1}T^{-i}\tau))=
H_{\bar{\mu}}(\bigvee_{i=0}^{n-1}{\bar{T}}^{-(i+i_*)}
\alpha|{\bar{T}}^{-i_*}(\bigvee_{i=0}^{n-1}\bar{T}^{-i}\Pi_X^{-1}\tau))\\
&=H_{\bar{\mu}}(\bigvee_{i=0}^{n-1}{\bar{T}}^{-(i+i_*)} \alpha|\bigvee_{j=0}^{n-1}\bar{T}^{-j}\gamma)\le
H_{\bar{\mu}}(\bigvee_{i=0}^{n+i_*-1}{\bar{T}}^{-i} \alpha|\bigvee_{j=0}^{n-1}\bar{T}^{-j}\gamma)\\ &\le
\sum_{i=0}^{n+i_*-1}H_{\bar{\mu}}(\bar{T}^{-i} \alpha|\bigvee_{j=0}^{n-1}\bar{T}^{-i}\gamma)\le \left(
\sum_{i=0}^{n-1}H_{\bar{\mu}}(\bar{T}^{-i}\alpha|\bar{T}^{-i}\gamma)
+\sum_{i=n}^{n+i_*-1}H_{\bar{\mu}}(\bar{T}^{-i}\alpha)\right) \\ &=\left( nH_{\bar{\mu}}(\alpha|\gamma)
+i_*H_{\bar{\mu}}(\alpha)\right). \end{align*} Using the above inequality, we have \begin{align*}
h_{\bar{\mu}}(\bar{T},\alpha|\Pi_X^{-1}(\mathcal{B}))&=\lim_{n\rightarrow +\infty}
\frac{1}{n}H_{\bar{\mu}}(\bigvee_{i=0}^{n-1}{\bar{T}}^{-i}\alpha|\Pi_X^{-1}(\mathcal{B})) \le
\lim_{n\rightarrow +\infty} \frac{1}{n} \left( nH_{\bar{\mu}}(\alpha|\gamma)
+i_*H_{\bar{\mu}}(\alpha)\right)\\ &=H_{\bar{\mu}}(\alpha|\gamma)<\frac{1}{m}. \end{align*}
 Since $m$ is arbitrary, $h_{\bar{\mu}}(\bar{T},\alpha|\Pi_X^{-1}(\mathcal{B}))=0$. This implies
 $h_{\bar{\mu}}(\bar{T}|\Pi_X)=0$ since
$\alpha$ is arbitrary. The proof is complete.
\hfill$\square$

\smallskip
Next, we investigate the ergodic decomposition of
conditional entropy. Let  $(X,\mathcal{B},\mu,T)$ be an invertible Lebesgue  system. We consider the sub-$\sigma$-algebra
$$
I_\mu(T)=\{  A\in \mathcal{B}: \mu(T^{-1}B\Delta B)=0\}.
$$ It is well known that there exists a factor map
$\phi:(X,\mathcal{B},\mu,T)\rightarrow (E,\mathcal{E},\eta,id_E)$ between two invertible Lebesgue systems
such that $\phi^{-1}(\mathcal{E})=I_\mu(T)$  (mod $\mu$)  and $(X,\mathcal{B}_e,\mu_e,T)$ is an ergodic
invertible Lebesgue system for $\eta$-a.e $e\in E$,  where $id_E$ is  the identity map from $E$ to itself,
$\mu=\int_E \mu_e d \eta(e)$ is the disintegration of $\mu$ relative to the  factor
$(E,\mathcal{E},\eta,id_E)$, and $(X,\mathcal{B}_e,\mu_e)$ is the corresponding Lebesgue space for $e\in  E$
in this disintegration  (see \cite{R}, \cite[Theorem 3.42]{G}).
More precisely, conditional probability measures $\{\mu_e\}_{e\in E}$ with the following
properties:
\begin{itemize} \item $\mu_e$  is a Lebesgue  measure on $X$ with $\mu_e(\phi^{-1}(e))=1$ for all
$y\in Y$.

\item for each $f \in L^1(X,\mathcal{B},\mu)$,  one  has $f \in L^1(X,\mathcal{B}_e,\mu_e)$ for $\eta$-a.e. $e\in E$, the map $e \mapsto
    \int_X f\,d\mu_e$ is in $L^1(E,\eta)$ and $\int_E \left(\int_X f\,d\mu_e \right)\, d\eta(e)=\int_X f
    \,d\mu$. Particularly, for each $A\in \mathcal{B}$, one has $A\in \mathcal{B}_e$ for $\eta$-a.e. $e\in E$.
\end{itemize}
The disintegration $\mu=\int_E \mu_e d
\eta(e)$ is called the  ergodic decomposition of $\mu$. The following result is well known (see \cite[Theorem
8.11]{R}, \cite{J} or \cite[Theorem 15.12]{G}).

\begin{lem}\label{edep} (Ergodic decomposition of entropy  for partition) Let $(X,\mathcal{B},\mu,T)$ be an
invertible Lebesgue system and $\mu=\int_E \mu_e d \eta(e)$ be  the  ergodic decomposition of $\mu$. Then for
any $\alpha\in \mathcal{P}_X(\mathcal{B})$, $$h_\mu(T,\alpha)=\int_E h_{\mu_e}(T,\alpha)  d \eta(e).$$
\end{lem}

\begin{lem}\label{relative-entropy} Let $(X,\mathcal{B},\mu, T)$ be a MDS.  Then for  any $\alpha,\beta\in
\mathcal{P}_X$, $$\inf_{n\ge  1}\frac{1}{n}
H_\mu(\bigvee_{i=0}^{n-1}T^{-i}\alpha|\bigvee_{i=0}^{n-1}T^{-i}\beta)=h_\mu(T,\alpha\vee
\beta)-h_\mu(T,\alpha).$$
\end{lem}
\begin{proof} For  $\alpha,\beta\in \mathcal{P}_X$,  let
$a_n=H_\mu(\bigvee_{i=0}^{n-1}T^{-i}\alpha|\bigvee_{i=0}^{n-1}T^{-i}\beta)$ for $n\in \mathbb{N}$. Then, for
$n,m\in \mathbb{N}$, we have
\begin{align*}
a_{n+m}&=H_\mu(\bigvee_{i=0}^{n+m-1}T^{-i}\alpha|\bigvee_{i=0}^{n+m-1}T^{-i}\beta)\\ &\le
H_\mu(\bigvee_{i=0}^{n-1}T^{-i}\alpha|\bigvee_{i=0}^{n+m-1}T^{-i}\beta)+H_\mu(T^{-n}(\bigvee_{i=0}^{m-1}T^{-i}\alpha)|\bigvee_{i=0}^{n+m-1}T^{-i}\beta)\\
&\le a_n+H_\mu(T^{-n}(\bigvee_{i=0}^{m-1}T^{-i}\alpha)|T^{-n}(\bigvee_{i=0}^{m-1}T^{-i}\beta))\\ &=a_n+a_m.
\end{align*} By Theorem 4.9 in \cite{W}, $\inf_{n\ge 1}\frac{1}{n}a_n=\lim_{n\rightarrow
+\infty}\frac{1}{n}a_n$. Thus, \begin{align*} &\hskip0.5cm \inf_{n\ge  1}\frac{1}{n}
H_\mu(\bigvee_{i=0}^{n-1}T^{-i}\alpha|\bigvee_{i=0}^{n-1}T^{-i}\beta)\\ &=\lim_{n\rightarrow
+\infty}\frac{1}{n} H_\mu(\bigvee_{i=0}^{n-1}T^{-i}\alpha|\bigvee_{i=0}^{n-1}T^{-i}\beta)\\
&=\lim_{n\rightarrow +\infty}\left( \frac{1}{n} H_\mu(\bigvee_{i=0}^{n-1}T^{-i}(\alpha\vee
\beta))-H_\mu(\bigvee_{i=0}^{n-1}T^{-i}\beta)\right)\\ &=\lim_{n\rightarrow +\infty}\frac{1}{n}
H_\mu(\bigvee_{i=0}^{n-1}T^{-i}(\alpha\vee \beta))-\lim_{n\rightarrow
+\infty}H_\mu(\bigvee_{i=0}^{n-1}T^{-i}\beta)\\ &=h_\mu(T,\alpha\vee \beta)-h_\mu(T,\beta). \end{align*} This
completes the proof of  Lemma. \end{proof}

Let  $(X,\mathcal{B},\mu,T)$ be an invertible Lebesgue  system.  A sub-$\sigma$-algebra $\mathcal{C}$ of
$\mathcal{B}$ is called {\it countably generated} if there exists $\{ A_i\}_{i=1}^\infty \subset \mathcal{C}$
such that $\mathcal{C}$ is the $\sigma$-algebra generated by $\{  A_i\}_{i=1}^\infty$, i.e., $\mathcal{C}$ is
the smallest $\sigma$-algebra containing all $A_i, i\in \mathbb{N}$.

\begin{lem}\label{cond-ed} (Ergodic decomposition of conditional entropy  for partition) Let
$(X,\mathcal{B},\mu,T)$ be an  invertible Lebesgue system, $\mu=\int_E \mu_e d \eta(e)$  be the  ergodic
decomposition of $\mu$, and $\mathcal{C}$ be a countably generated  sub-$\sigma$-algebra of $\mathcal{B}$ with
$T^{-1}\mathcal{C}\subseteq \mathcal{C}$. Then, for each $\alpha\in \mathcal{P}_X(\mathcal{B})$,
$$h_\mu(T,\alpha|\mathcal{C})=\int_E h_{\mu_e}(T,\alpha|\mathcal{C}) d \eta(e).$$
\end{lem}
\begin{proof} Let
$(X,\mathcal{B}_e,\mu_e,T)$ be the corresponding Lebesgue systems for $e\in  E$ in the
 ergodic decomposition of $\mu$. Since $\mathcal{C}$ is a countably  generated  sub-$\sigma$-algebra of
 $\mathcal{B}$,
 there exists $\{ A_i\}_{i=1}^\infty \subset \mathcal{C}$ such that $\mathcal{C}$ is the $\sigma$-algebra
 generated by
 $\{  A_i\}_{i=1}^\infty$. Let  $\alpha=\{B_1,B_2,\cdots,B_k\}\in \mathcal{P}_X(\mathcal{B})$.
 Since  $A\in \mathcal{B}_e$ for $\eta$-a.e.  $e\in E$ when $A\in  \mathcal{B}$,
 one  has that for
 $\eta$-a.e.  $e\in E$,
$\{B_1,B_2,\cdots,B_k\} \cup \{A_i:i\in \mathbb{N}\}\subseteq\mathcal{B}_e$. Thus for $\eta$-a.e.  $e\in E$,
$\mathcal{B}_e$ contains the $\sigma$-algebra generated by $\mathcal{C}\cup  (\bigcup_{i\in \mathbb{Z}}
T^{i}\{B_1,B_2,\cdots,B_k\} )$ since  $\mathcal{B}_e$ is $\sigma$-algebra and $T^{-1}\mathcal{B}_e=\mathcal{B}_e$.

Let $\beta_j=\bigvee_{i=1}^j \{  A_i,X\setminus  A_i\}$  for  $j\in \mathbb{N}$. Then for each $n\in
\mathbb{N}$, when $j\rightarrow +\infty$, we have $\bigvee_{i=0}^{n-1}T^{-i}\beta_j\nearrow \mathcal{C}$ (mod
$\mu$)  and $\bigvee_{i=0}^{n-1}T^{-i}\beta_j\nearrow \mathcal{C}$ (mod $\mu_e$) for $\eta$-a.e. $e\in E$.
Thus,
\begin{align} H_\mu(\bigvee_{i=1}^{n-1}T^{-i}\alpha|\mathcal{C})=\lim_{j\rightarrow
+\infty}H_\mu(\bigvee_{i=1}^{n-1}T^{-i}\alpha|\bigvee_{i=0}^{n-1}T^{-i}\beta_j)=\inf_{j\ge
1}H_\mu(\bigvee_{i=1}^{n-1}T^{-i}\alpha|\bigvee_{i=0}^{n-1}T^{-i}\beta_j).
\end{align}
Moveover,
\begin{align*} h_\mu(T,\alpha|\mathcal{C})&=\inf_{n\ge 1}
\frac{1}{n}H_\mu(\bigvee_{i=1}^{n-1}T^{-i}\alpha|\mathcal{C})=\inf_{n\ge 1}\left(\inf_{j\ge
1}\frac{1}{n}H_\mu(\bigvee_{i=0}^{n-1}T^{-i}\alpha|\bigvee_{i=0}^{n-1}T^{-i}\beta_j)\right)\\ &=\inf_{j\ge
1}\left(\inf_{n\ge
1}\frac{1}{n}H_\mu(\bigvee_{i=1}^{n-1}T^{-i}\alpha|\bigvee_{i=0}^{n-1}T^{-i}\beta_j)\right)\\ &=\inf_{j\ge
1}\left(h_\mu(T,\alpha\vee \beta_j)-h_\mu(T,\beta_j)\right).   \  \  \text{(by lemma \ref{relative-entropy})}
\end{align*}
That is
\begin{align}\label{limit-1}
\begin{aligned} h_\mu(T,\alpha|\mathcal{C})&=\inf_{j\ge
1}\left(\inf_{n\ge
1}\frac{1}{n}H_\mu(\bigvee_{i=1}^{n-1}T^{-i}\alpha|\bigvee_{i=0}^{n-1}T^{-i}\beta_j)\right)\\ &=\inf_{j\ge
1}\left(h_\mu(T,\alpha\vee \beta_j)-h_\mu(T,\beta_j)\right).
\end{aligned}
\end{align} Similarly,  for
$\eta$-a.e. $e\in E$, we have
\begin{align}\label{limit-2}
\begin{aligned}
h_{\mu_e}(T,\alpha|\mathcal{C})&=\inf_{j\ge 1}\left(\inf_{n\ge
1}\frac{1}{n}H_\mu(\bigvee_{i=1}^{n-1}T^{-i}\alpha|\bigvee_{i=0}^{n-1}T^{-i}\beta_j)\right)\\ &=\inf_{j\ge
1}\left(h_{\mu_e}(T,\alpha\vee \beta_j)-h_{\mu_e}(T,\beta_j)\right).
\end{aligned}
\end{align} Using
\eqref{limit-1}, Lemma \ref{edep}, Lemma \ref{relative-entropy},  and \eqref{limit-2}, we obtain
\begin{align*}
h_\mu(T,\alpha|\mathcal{C})&=\inf_{j\ge 1}\left(h_\mu(T,\alpha\vee \beta_j)-h_\mu(T,\beta_j)\right)\\
&=\inf_{j\ge 1}\int_E \left(h_{\mu_e}(T,\alpha\vee \beta_j)-h_{\mu_e}(T,\beta_j)\right)d \eta(e)\\
&=\inf_{j\ge 1}\int_E \inf_{n\ge
1}\frac{1}{n}H_{\mu_e}(\bigvee_{i=1}^{n-1}T^{-i}\alpha|\bigvee_{i=0}^{n-1}T^{-i}\beta_j) d \eta(e)\\
&=\lim_{j\rightarrow +\infty}\int_E \inf_{n\ge
1}\frac{1}{n}H_{\mu_e}(\bigvee_{i=1}^{n-1}T^{-i}\alpha|\bigvee_{i=0}^{n-1}T^{-i}\beta_j) d \eta(e)\\ &=\int_E
\lim_{j\rightarrow +\infty}\left(\inf_{n\ge
1}\frac{1}{n}H_{\mu_e}(\bigvee_{i=1}^{n-1}T^{-i}\alpha|\bigvee_{i=0}^{n-1}T^{-i}\beta_j)\right) d \eta(e) \\
&\hskip1cm \ \text{(by Monotone  convergence Theorem)}\\ &=\int_E \inf_{j\ge 1}\left(\inf_{n\ge
1}\frac{1}{n}H_{\mu_e}(\bigvee_{i=1}^{n-1}T^{-i}\alpha|\bigvee_{i=0}^{n-1}T^{-i}\beta_j) \right)d \eta(e)\\
&=\int_E h_{\mu_e}(T,\alpha|\mathcal{C}) \, d \eta(e).
\end{align*} This completes  the  proof of Lemma.
\end{proof}

Let $(X,\mathcal{B}_X,\mu,T)$ be a Polish system. Let $\mathcal{B}_\mu$ be the completion of the Borel
$\sigma$-algebra $\mathcal{B}_X$ with  respect to $\mu$.  Then $(X,\mathcal{B}_\mu,\mu,T)$ is a Lebesgue
system. Put $$\bar{X}=\{ \bar{x}=(x_i)_{i\in  \mathbb{Z}}\in  X^{\mathbb{Z}}: \, Tx_{i}=x_{i+1}, i\in
\mathbb{Z}\}$$ and  let $\Pi_X:(\bar{X},\bar{\mathcal{B}_\mu},\bar{\mu},\bar{T})\rightarrow
(X,\mathcal{B}_X,\mu,T)$ be the natural extension of $(X,\mathcal{B}_X,\mu,T)$.

Let $\bar{\mu}=\int_E \bar{\mu}_e d \eta(e)$ be the  ergodic decomposition of $\bar{\mu}$ and
$(\bar{X},(\bar{\mathcal{B}_\mu})_e,\bar{\mu}_e,\bar{T})$ be the corresponding Lebesgue systems for $e\in  E$
 in
the ergodic decomposition of $\bar{\mu}$. Since $\Pi_X^{-1}(\mathcal{B}_X)$ is a countably generated
sub-$\sigma$-algebra of $\mathcal{B}_{\bar{\mu}}$, one  knows that for $\eta$-a.e. $e\in E$,
$\Pi_X^{-1}(\mathcal{B}_X)\subset (\bar{\mathcal{B}_{\mu}})_e$.
Therefore, for $\eta$-a.e. $e\in E$, letting
$\mu_e(A)=\bar{\mu}_e(\Pi_X^{-1}A)$ for $A\in \mathcal{B}_X$, we have that  $(X,\mathcal{B}_X,\mu_e,T)$ is an ergodic
Polish system. It is clear that for any $f\in L^1(X,\mathcal{B}_X,\mu)$, one has
that $f\circ \Pi_X\in L^1(\bar{X},\Pi_X^{-1}(\mathcal{B}_X),\bar{\mu})$. Hence the map $e\in
E  \mapsto\int_X f(x)  d \mu_e(x)(:=\int_{\bar{X}}f\circ \Pi_X d \bar{\mu}_e)$  is  $\eta$-measurable  and
$$\int_E (\int_X f d  \mu_e) d \eta(e)=\int_E (\int_{\bar{X}}f\circ \Pi_X d \bar{\mu}_e)
d\eta(e)=\int_{\bar{X}} f\circ \Pi_X d \bar{\mu}=\int_X f d \mu.$$ That is, for any  $f\in
L^1(X,\mathcal{B}_X,\mu)$, one has \begin{equation}\label{eq-ed-key} \text{the map }e\in E\mapsto \int_X f(x)
d \mu_e(x)\text{ is $\eta$-measurable  and } \int_E (\int_X f d  \mu_e)  d \eta(e)=\int_X f d \mu.
\end{equation} In this case,  we say that $\mu=\int_E \mu_e d \eta(e)$ is {\it the ergodic decomposition  of
$\mu$}.

\begin{lem}\label{ergodic} Let $\pi:(X,\mathcal{B}_X,\mu,T)\rightarrow (Y,\mathcal{B}_Y,\nu,S)$ be a factor
map between two Polish systems and $(Y,\mathcal{B}_Y,\nu,S)$  be ergodic.  If $\mu=\int_E \mu_e d \eta(e)$ is
the ergodic decomposition  of $\mu$, then

\begin{enumerate} \item  for $\eta$-a.e. $e\in  E$, $\pi:(X,\mathcal{B}_X,\mu_e,T)\rightarrow
(Y,\mathcal{B}_Y,\nu,S)$ is
    a factor map between two Polish systems.

\item $h_\mu(T|\pi^{-1}(\mathcal{B}_Y))=\int_E h_{\mu_e}(T|\pi^{-1}(\mathcal{B}_Y)) d \eta(e)$. That is,
    $h_\mu(T|\pi)=\int_E h_{\mu_e}(T|\pi) d \eta(e)$. \end{enumerate} \end{lem} \begin{proof}Let
    $\mathcal{B}_\nu$ be the completion of the Borel $\sigma$-algebra $\mathcal{B}_Y$ with respect to
    $\nu$. Let $\mathcal{B}_\mu$ be the completion of the Borel $\sigma$-algebra $\mathcal{B}_X$ with
    respect to $\mu$. Then $\pi:(X,\mathcal{B}_\mu,\mu, T)\rightarrow (Y,  \mathcal{B}_\nu,\nu,S)$  is a
    factor map between  two Lebesgue  systems.

Put $$\bar{X}=\{ \bar{x}=(x_i)_{i\in  \mathbb{Z}}\in  X^{\mathbb{Z}}: \, Tx_{i}=x_{i+1}, i\in \mathbb{Z}\}$$
and  let $\Pi_X:(\bar{X},\bar{\mathcal{B}_\mu},\bar{\mu},\bar{T})\rightarrow (X,\mathcal{B}_X,\mu,T)$ be the
natural extension of $(X,\mathcal{B}_X,\mu,T)$. Let $\bar{\mu}=\int_E \bar{\mu}_e d \eta(e)$  the  ergodic
decomposition of $\bar{\mu}$ and $(\bar{X},(\bar{\mathcal{B}_\mu})_e,\bar{\mu}_e,\bar{T})$ be the corresponding
Lebesgue systems for $e\in  E$ in the ergodic decomposition of $\bar{\mu}$. Since $\Pi_X^{-1}(\mathcal{B}_X)$
is a countably generated sub-$\sigma$-algebra of $\mathcal{B}_{\bar{\mu}}$, one  knows that for $\eta$-a.e.
$e\in E$, $\Pi_X^{-1}(\mathcal{B}_X)\subset (\bar{\mathcal{B}_{\mu}})_e$. For $\eta$-a.e. $e\in E$, letting
$\mu_e(A)=\bar{\mu}_e(\Pi_X^{-1}A)$ for $A\in \mathcal{B}_X$, then $(X,\mathcal{B}_X,\mu_e,T)$ is an ergodic
Polish system and $\mu=\int_E \mu_e d \eta(e)$ is the ergodic decomposition  of
$\mu$.

For $\eta$-a.e. $e\in E$, we define $\nu_e(B)=\mu_e(\pi^{-1}(B))$ for $B\in \mathcal{B}_Y$. Then
$(Y,\mathcal{B}_Y,\nu_e,S)$ is an ergodic  Polish system and $\pi:(X,\mathcal{B}_X,\mu_e,T)\rightarrow
(Y,\mathcal{B}_Y,\nu_e,S)$ is a factor map between two Polish  systems. Thus, the property (1) in the lemma follows  from the following
claim.

\medskip \noindent{\bf Claim:} For  $\eta$-a.e. $e\in E$, $\nu_e=\nu$.

\begin{proof}[Proof of Claim] Since $Y$ is a Polish space, there are  finite Borel-measurable partitions
$\beta_i$, $i\in \mathbb{N}$ of $Y$ such that  $\beta_1\preceq  \beta_2\preceq\cdots$ and $\mathcal{B}_Y$ is
the smallest $\sigma$-algebra containing all $\beta_i$,  $i\in \mathbb{N}$. Let $\mathcal{D}$ be the algebra
generated  by $\{  A:A\in \beta_i \text{ for some }i\in \mathbb{N}\}$. Then  $\mathcal{D}$  is a countable
set and $\mathcal{D}$ generates the $\sigma$-algebra $\mathcal{B}_Y$. Define $$Y(\nu)=\{ y\in Y:
\lim_{n\rightarrow +\infty}  \frac{1}{n}\sum_{i=0}^{n-1}1_A(S^iy)=\nu(A) \text{ for any }A\in \mathcal{D}\}$$
where $1_A$ is the characterization  function of $A$. Since $\mathcal{D}$  is a countable set,  $Y(\nu)\in
\mathcal{B}_Y$  and $\nu(Y(\nu))=1$ by the Birkhoff  ergodic theorem. Since $\int_E \mu_e(\pi^{-1}(Y(\nu))d
\eta(e)=\mu(\pi^{-1}(Y(\nu))=\nu(Y(\nu))=1$, one  has $\mu_e(\pi^{-1}(Y(\nu)))=1$  for  $\eta$-a.e. $e\in E$.
That is, $\nu_e(Y(\nu))=\mu_e(\pi^{-1}(Y(\nu)))=1$  for  $\eta$-a.e. $e\in E$. Thus, to show $\nu_e=\nu$ for
$\eta$-a.e. $e\in E$, it is sufficient to  show that if $e\in E$ such that $(Y,\mathcal{B}_Y,\nu_e,S)$ is an
ergodic  Polish system  and $\nu_e(Y(\nu))=1$, then $\nu_e=\nu$.

Let $e\in E$ such that $(Y,\mathcal{B}_Y,\nu_e,S)$ is an ergodic  Polish system  and $\nu_e(Y(\nu))=1$. Set
$$\mathcal{F}_e=\{  B\in  \mathcal{B}_Y: \nu_e(B)=\nu(B)\}.$$ We want to show
$\mathcal{F}_e=\mathcal{B}_Y$.

By the Birkhoff ergodic theorem, there exists $y\in Y(\nu)$ such that
$$\lim_{n\rightarrow +\infty}
\frac{1}{n}\sum_{i=0}^{n-1}1_A(S^iy)=\nu_e(A) \text{ for any }A\in \mathcal{D}.
$$ By the definition of
$Y(\nu)$, $\lim_{n\rightarrow +\infty}  \frac{1}{n}\sum_{i=0}^{n-1}1_A(S^iy)=\nu(A)$ for  any $A\in
\mathcal{D}$. Hence,   $\nu_e(A)=\nu(A)$ for  any $A\in \mathcal{D}$, which  implies that
$\mathcal{D}\subseteq \mathcal{F}_e$. Note that  $\mathcal{F}_e$ is a monotone class. Thus,   $\mathcal{F}_e=\mathcal{B}_Y$  follows from that  the $\sigma$-algebra generated  by $\mathcal{D}$ is the
monotone class generated by $\mathcal{D}$. This  completes the proof of Claim. \end{proof}

Next,  we prove the property  (2).  Since $X$ is  Polish space, there are finite Borel-measurable partitions
$\alpha_i$ of $X$ such that  $\alpha_1\preceq\alpha_2\preceq\cdots$ and the small $\sigma$-algebra containing
all $\alpha_i$,  $i\in \mathbb{N}$ is  $\mathcal{B}_X$.  Thus for any a $T$-invariant Borel probability
measure $\lambda$ on $(X,\mathcal{B}_X)$ and any sub-$\sigma$-algebra $\mathcal{C}$ of $\mathcal{B}_X$ with
$T^{-1}\mathcal{C}\subseteq \mathcal{C}$, one has \begin{align}\label{p-limit} \lim_{i\rightarrow
+\infty}h_\lambda(T,\alpha_i|\mathcal{C})=h_\theta(T|\mathcal{C}). \end{align}

Since $\Pi_X^{-1}(\pi^{-1}\mathcal{B}_Y)$  is a countably generated
sub-$\sigma$-algebra of $\mathcal{B}_X$ and
$$\bar{T}^{-1}(\Pi_X^{-1}(\pi^{-1}\mathcal{B}_Y))=\Pi_X^{-1}(T^{-1}(\pi^{-1}\mathcal{B}_Y)=\Pi_X^{-1}(\pi^{-1}(S^{-1}\mathcal{B}_Y))\subseteq
\Pi_X^{-1}(\pi^{-1}\mathcal{B}_Y),
$$  one has for  $i\in \mathbb{N}$
\begin{equation}\label{eq-kkky}
h_{\bar{\mu}}(\bar{T},\Pi_X^{-1}(\alpha_i)|\Pi_X^{-1}(\pi^{-1}\mathcal{B}_Y))=\int_E
h_{\bar{\mu}_e}(\bar{T},\Pi_X^{-1}(\alpha_i)|\Pi_X^{-1}(\pi^{-1}\mathcal{B}_Y)d \eta(e) \end{equation} by
lemma \ref{cond-ed}. Note that
$$h_\mu(T,\alpha_i|\pi^{-1}\mathcal{B}_Y)=h_{\bar{\mu}}(\bar{T},\Pi_X^{-1}(\alpha_i)|\Pi_X^{-1}(\pi^{-1}\mathcal{B}_Y))$$
and
$$h_{\mu_e}(T,\alpha_i|\pi^{-1}\mathcal{B}_Y)=h_{\bar{\mu}_e}(\bar{T},\Pi_X^{-1}(\alpha_i)|\Pi_X^{-1}(\pi^{-1}\mathcal{B}_Y))$$
for $\eta$-a.e. $e\in E$. Combing  this with \eqref{eq-kkky}, one has \begin{equation}\label{i-eq-lim}
h_\mu(T,\alpha_i|\pi^{-1}\mathcal{B}_Y)=\int_E h_{\mu_e}(T,\alpha_i|\pi^{-1}\mathcal{B}_Y)d \eta(e).
\end{equation} Let $i\rightarrow +\infty$ in \eqref{i-eq-lim}, we have
$$h_\mu(T|\pi^{-1}\mathcal{B}_Y)=\int_E h_{\mu_e}(T|\pi^{-1}\mathcal{B}_Y)d \eta(e)$$ by \eqref{p-limit} and
the monotone convergence Theorem. The proof of this lemma is complete. \end{proof}

\end{document}